\documentclass{amsart}

\usepackage{amsthm,amssymb,latexsym,amsmath}
\usepackage[all]{xy}
\usepackage[english]{babel}
\usepackage[utf8x]{inputenc}

\newcommand{\CC}{{\mathbb C}}
\newcommand{\p}[1]{{\mathbb{P}^{#1}}}

\newcommand{\osigma}{{\mathcal O}_{\Sigma}}
\newcommand{\PP}{{\mathbb P}}
\newcommand{\op}[1]{{\mathcal O}_{\mathbb{P}^{#1}}}
\newcommand{\opn}{{\mathcal O}_{\mathbb{P}^n}}
 \newcommand{\ox}{{\mathcal O}_{X}}
\newcommand{\ol}{{\mathcal O}_{\ell}}
\newcommand{\pr}{\operatorname{pr}\nolimits}
\newcommand\pt{{\mathrm{\hspace{0.2ex}p\hspace{-0.23ex}t}}}
\newcommand{\C}{\mathbb{C}}
\newcommand{\cald}{{\mathcal D}}
\newcommand{\calr}{{\mathcal R}}
\newcommand{\calq}{{\mathcal Q}}
\newcommand{\calb}{{\mathcal B}}

\newcommand{\cale}{{\mathcal E}}

\newcommand{\cali}{{\mathcal I}}
\newcommand{\calj}{{\mathcal J}}

\newcommand{\calm}{{\mathcal M}}
\newcommand{\calo}{{\mathcal O}}
\newcommand{\calp}{{\mathcal P}}

\newcommand{\calz}{{\mathcal Z}}

\newcommand{\HOM}{{\mathcal H}{\it om}}
\newcommand{\inhom}{{\mathcal H}{\it om}}
\newcommand{\inext}{{\mathcal E}{\it xt}}
\newcommand{\intor}{{\mathcal T\!}{\it or}}
\newcommand{\simto}{\stackrel{\sim}{\to}}

\DeclareMathOperator{\ext}{Ext}
\DeclareMathOperator{\Hom}{Hom}

\DeclareMathOperator{\coker}{coker}
\DeclareMathOperator{\im}{im}

\newcommand\dual{{\scriptscriptstyle{\vee}}}

\newcommand\lra{{\longrightarrow}}

\newcommand{\into}{\hookrightarrow}
\newcommand{\onto}{\twoheadrightarrow}

\newlength{\rrrr}

\newcommand{\isoto}{{\lra\hspace{-1.3 em}
\raisebox{ 0.6 ex}{$\textstyle\sim$}\hspace{0.8 em}}}

\newcommand{\intoo}[1]{\:
\xymatrix@1{\ar@{^(->}[r]^{#1}&}\:}
\newcommand{\ontoo}[1]{\:
\xymatrix@1{\ar@{->>}[r]^{#1}&}\:}

\newcommand\Ext{\operatorname{Ext}\nolimits}

\newcommand\smallvee{{\scriptscriptstyle{\vee}}}

\newtheorem{theorem}{Theorem}[section]

\newtheorem{proposition}[theorem]{Proposition}
\newtheorem{lemma}[theorem]{Lemma}
\newtheorem{corollary}[theorem]{Corollary}

\newtheorem{remark}[theorem]{Remark}

\begin{document}

\title{New divisors in the boundary of the instanton moduli space}

\author{Marcos Jardim}
\address{IMECC - UNICAMP \\
Departamento de Matem\'atica \\ Caixa Postal 6065 \\
13083-970 Campinas-SP, Brazil}
\email{jardim@ime.unicamp.br}

\author{Dimitri Markushevich}
\address{Math\'ematiques -- b\^{a}t.~M2\\
Universit\'e Lille 1\\
F-59655 Villeneuve d'Ascq Cedex, France}
\email{markushe@math.univ-lille1.fr}

\author{Alexander S. Tikhomirov}
\address{Department of Mathematics\\
Higher School of Economics\\
7 Vavilova Street\\ 
117312 Moscow, Russia}
\email{astikhomirov@mail.ru}

\begin{abstract}
Let $\cali(n)$ denote the moduli space of rank $2$ instanton bundles of charge $n$ on $\p3$. We know from \cite{CO,JV,T1,T2} that $\cali(n)$ is an irreducible, nonsingular and affine variety of dimension
$8n-3$. Since every rank $2$ instanton bundle on $\p3$ is stable, we may regard $\cali(n)$ as an open subset of the projective Gieseker--Maruyama moduli scheme $\calm(n)$ of rank $2$ semistable torsion free sheaves
$F$ on $\p3$ with Chern classes $c_1=c_3=0$ and $c_2=n$, and consider the closure $\overline{\cali(n)}$ of
$\cali(n)$ in $\calm(n)$.

We construct some of the irreducible components of dimension $8n-4$ of the boundary 
$\partial\cali(n):=\overline{\cali(n)}\setminus\cali(n)$. These components generically lie in the smooth locus of $\calm(n)$ and consist of rank $2$ torsion free instanton sheaves with singularities along rational curves.
\end{abstract}
\maketitle

\section{Introduction}

A {\it mathematical instanton of charge $n$} is a holomorphic rank $2$
vector bundle $E$ on the projective space $\mathbb{P}^3$ with Chern classes
\begin{equation}\label{Chern classes}
c_1(E)=0,\ \ \ c_2(E)=n,
\end{equation}
satisfying the vanishing conditions
\begin{equation}\label{vanishing condns}
h^0(E(-1))=h^1(E(-2))=0.
\end{equation}
The epithet ``mathematical'', which will be omitted in the remainder of the paper, distinguishes these objects from {\it physical} instantons. The latter are anti-self-dual $SU(2)$-connections on $S^4$,  which give rise, by the Atiyah--Ward correspondence, to vector bundles $E$ as above with some additional ``reality'' conditions.

We denote by $\cali(n)$ the moduli space of instantons of charge $n$. Nowadays it is known that $\cali(n)$ is a nice variety possessing some natural properties which were conjectured long ago, but whose proof remained an open problem for many years: $\cali(n)$ is affine \cite{CO}, nonsingular \cite{JV}, and irreducible of dimension $8n-3$ \cite{T1,T2}. Since every instanton is $\mu$-stable and $\mu$-stability is the same as Gieseker stability for rank $2$ bundles, $\cali(n)$ can be regarded as an open subset within the projective Maruyama moduli space $\calm(n)$ of Gieseker semistable rank $2$ sheaves $F$ on $\p3$ with $c_1(F)=c_3(F)=0$ and $c_2(F)=n$.

Let $\overline{\cali(n)}$ denote the closure of $\cali(n)$ within $\calm(n)$. Our goal is to approach the understanding of the boundary
$\partial\cali(n):=\overline{\cali(n)}\setminus\cali(n)$ of this compactification.
Partial results in this direction are already known; in particular, it is clear that $\partial\cali(n)$ has, in general, several irreducible components. These components have been completely determined only for charges $c_2=1$ \cite{B} and $c_2=2$ \cite{NT}. The case $c_2=3$ was partially treated in \cite{GS,Per1,Per3}. Perrin in \cite[Remarque 3.6.8]{Per3} cites a conjecture of Gruson and Trautmann, suggesting that the boundary of $\cali(3)$ consists of 8 irreducible divisors. He also advances towards a proof of this conjecture by constructing two out of the eight conjectural boundary divisors, in addition to the three ones constructed before.

All the boundary components for charges $c_2\leq 3$, known or conjectural, are divisorial and consist of {\em non locally free sheaves}. A new phenomenon occurs for $c_2=5$: in \cite{Rao},
Rao shows the existence of a divisorial component of $\partial\cali(5)$ whose generic point represents a {\em vector bundle}, for which the vanishing condition $h^1(E(-2))=0$ fails.

In the present paper we describe $n$ irreducible components of $\partial\cali(n)$, whose generic points represent instanton {\em sheaves}. By definition,
these are rank 2 torsion free sheaves $E$ on $\mathbb P^3$ satisfying
\eqref{Chern classes}, \eqref{vanishing condns}, but also the additional vanishing conditions
$h^{2}(E(-2))=h^3(E(-3))=0$, which, by Serre duality, are automatically satisfied when $E$ is locally free.
Thus, for example, Rao's component of $\partial\cali(5)$ does not satisfy
these vanishing conditions and is not of the type we are interested in. Also one can see that among the eight Gruson--Trautmann (partly conjectural) components of $\partial\cali(3)$, only four parametrize instanton sheaves in our terminology. One can say that the instanton sheaves provide a {\em partial} compactification of $\cali(n)$, and we study the boundary of this partial compactification.

Our main result is that for any $n\geq 1$, $\partial\cali(n)$ contains $n$ irreducible divisors $\overline{\cald(m,n)}$, where $m=1,\ldots,n$, of $\overline{\cali(n)}$, whose generic points represent instanton sheaves and are smooth points of $\calm(n)$. A sheaf in $\overline{\cald(m,n)}$, generically, is singular along a normal rational curve of degree $m$. By \cite{JG}, the singular locus of any instanton sheaf, if nonempty,  is a scheme of pure dimension 1. An important invariant of a non locally free instanton sheaf $E$ is $Q_E:=E^{\smallvee\smallvee}/E$, which is a pure 1-dimensional sheaf supported on the singular locus of $E$. A necessary condition for a pure 1-dimensional sheaf $Q$ to be $Q_E$ for some $E\in \partial\cali(n)$ was determined in \cite[Thm. 0.1]{Per4}: $Q(-2)$ should be a theta-characteristic of its supporting 1-dimensional scheme. 

Our boundary components $\overline{\cald(m,n)}$ coincide with some of the components considered in \cite{B} (case $n=1$), \cite{NT} (case $n=2$) and \cite{Per1,Per3,Per4} (case $n=3$). For $n>3$, the components here presented are new. However, our components  are not sufficient to exhaust the boundary of the partial compactification by instanton sheaves, for one\sloppy\ of the components of $\partial\cali(3)$ presented in \cite{Per3} is associated to theta characteristics of plane cubic curves in $\mathbb P^3$. We still do not know which pure 1-dimensional sheaves do occur as singular locus of non locally free instanton sheaves in the boundary of $\cali(n)$. Nonetheless, we expect that our examples, associated to theta characteristics of smooth rational curves, are singled out by that they are smooth points of $\calm(n)$. We provide an evidence to support this expectation by proving that the instanton sheaves which are singular along elliptic curves of degree $\leq 4$, form a separate component of $\calm(n)$. 

Now we will briefly describe the content of the paper by sections. 
In Section 2, we introduce instanton sheaves and provide some relevant properties. In Section 3, we prove basic properties of the main tool for constructing non locally free instanton sheaves: elementary transformations along a curve in $\mathbb P^3$ endowed with a line bundle $Q$ satisfying $h^0(Q(-2))=h^1(Q(-2))=0$. In Section 4 we prove the stability of the sheaves obtained
by elementary transformations from instanton bundles. Section 5 contains calculations with local and global exts involving a sheaf $F$, obtained by an elementary transformation, with a view towards computing the infinitesimal deformation and obstruction spaces $\Ext^i(F,F)\ (i=1,2)$.

In Section 6 we construct the irreducible $(8n-4)$-dimensional irreducible
subvarieties $\overline{\cald(m,n)}$ of $\calm(n)$ and their open subsets $\cald(m,n)$, parametrizing instanton sheaves $F$ obtained by elementary transformations of instanton bundles of charge $n$ along rational normal curves of degree $m$ (Proposition \ref{variety D}). We also prove that for an instanton sheaf $F$ from $\cald(m,n)$, we have $\dim\Ext^1(F,F)=8n-3$ and $\Ext^2(F,F)=0$, so that $\calm(n)$ is smooth of local dimension $8n-3$ at $F$ (Proposition \ref{prop 4.1}).

In Section 7, we prove the main result of the paper, stating that
$\overline{\cald(m,n)}$ are divisors lying in the boundary
$\partial\cali(n)$ of the instantons of charge $n$ for each $m=1,\ldots,n$ (Theorem \ref{Thm 4.6}).
This is done by induction on $m$ and $n$, using deformations of elementary transforms of several types along 1-dimensional sheaves with supports on reducible rational curves.

In Section 8, we show that the elementary transformations of null-correlation instanton bundles along elliptic quartic curves provide a generically smooth irreducible component
of $\calm(5)$ of dimension 37, equal to the dimension of $\cali(5)$ (Theorem \ref{thm-quartics}). We also remark that somewhat easier arguments prove the following similar result: the elementary transforms of $\mathcal O_{\mathbb P^3}^{\oplus 2}$ along plane cubic curves fill an irreducible component of $\calm(3)$ of dimension 21, equal to the dimension of $\cali(3)$. This property implies, in particular, that the boundary divisor of $\cali(3)$ described by Perrin in \cite{Per3}, associated to plane cubic curves, is a non normal singularity of $\calm(3)$, given by the intersection of this new irreducible component with $\overline{\cali(3)}$.

Throughout the paper, the base field will be the field of complex numbers $\CC$. We work in the algebraic setting, so the term ``variety" means ``reduced scheme of finite type over $\CC$". We will use interchangeably the terms ``vector bundle" and ``locally free sheaf". The
projectivization $\PP(V)$ of a vector space $V$ will be understood as the variety of lines
in $V$, and the  projectivization $\PP(E)$ of a vector bundle $E$ over a variety $X$ as the relative proj-scheme Proj$_X\big(\mathcal{S}ym_{\scriptscriptstyle\bullet}(E^\dual)\big)$.\medskip

\noindent {\bf Acknowledgements.}
MJ is partially supported by the CNPq grant number 302477/2010-1 and the FAPESP grant number 2014/14743-8. MJ aknowledges the hospitality of the University of Lille 1. DM acknowledges the FAPESP grant number 2011/06464-3, which allowed him to visit the University of Campinas. DM was also partially supported by Labex CEMPI (ANR-11-LABX-0007-01). DM and AT acknowledge the hospitality of the Max-Planck-Institut für Mathematik in Bonn, where they made a part of work on the paper.

\section{Monads and instanton sheaves}

In this section we set up notation and revise important results from the vast literature on instanton sheaves on $\p3$. We also establish a couple of new propositions that will be relevant for the main results of the present paper. We work over the field of complex numbers; all sheaves of modules are coherent.

Recall that a \emph{monad} on a projective variety $X$ is a complex of locally free sheaves of the form
\begin{equation}\label{g-monad}
A \stackrel{\alpha}{\rightarrow} B \stackrel{\beta}{\rightarrow} C
\end{equation}
where $\alpha$ is injective and $\beta$ is surjective. The sheaf $\ker\beta/\im\alpha$ is called the \emph{cohomology} of (\ref{g-monad}). If $A=\ox(-1)^{\oplus a}$, $B=\ox^{\oplus b}$ and
$C=\ox(1)^{\oplus c}$, then (\ref{g-monad}) is a called \emph{linear monad}.

An \emph{instanton sheaf} on $\p3$ is a torsion free sheaf $F$ with trivial determinant and satisfying the following cohomological conditions
$$ H^0(F(-1))=H^1(F(-2))=H^{2}(F(-2))=H^3(F(-3))=0. $$
The integer $n:=-\chi(F(-1))$ is called the charge of $F$; it is easy to check that $n=h^1(F(-1))=c_2(F)$. The trivial sheaf $\op3^{\oplus r}$ of rank $r$ is considered as an instanton sheaf of charge zero.

Using the Beilinson spectral sequence, one can show that the instanton sheaves are precisely those that arise as cohomologies of linear monads of the form
$$
H^1(F\otimes\Omega^2_{\p3}(1))\otimes\op3(-1) \to H^1(F\otimes\Omega^1_\p3)\otimes\op3
\to H^1(F(-1))\otimes\op3(1),
$$
see for instance \cite[Proposition 14]{FJ2}. One checks that
$h^1(F\otimes\Omega^2_{\p3}(1))=h^1(F(-1))=n$ and $h^1(F\otimes\Omega^1_\p3)=2n+r$, where $r$ denotes the rank of $F$. The above monad can therefore be written in the following simpler way:
\begin{equation} \label{monad1}
M^\bullet ~:~ \op3(-1)^{\oplus n} \stackrel{\alpha}{\rightarrow} \op3^{\oplus 2n+r}
\stackrel{\beta}{\rightarrow} \op3(1)^{\oplus n}
\end{equation}

The following result will be relevant in what follows.

\begin{lemma}\label{exts}
If $F$ is an instanton sheaf on $\p3$, then
\begin{itemize}
\item[(i)] $\inext^p(F,F)=0$ for $p\ge2$;
\item[(ii)] $\inext^p(F,\op3)=0$ for $p\ge2$;
\item[(iii)] $\ext^3(F,F)=0$.
\end{itemize}
\end{lemma}

\begin{proof}
Since $F$ can be realized as the cohomology of a monad of the form (\ref{monad1}), we have a short exact sequence
$$
0 \to \op3(-1)^{\oplus c} \stackrel{\alpha}{\rightarrow} K \to F \to 0,
$$
where $K:=\ker\beta$ is a locally free sheaf. Applying the functor $\inhom(-,F)$ we obtain
$$ \inext^{p-1}(\opn(-1)^{\oplus a},F) \to \inext^{p}(F,F) \to \inext^p(K,F) . $$
Since $\op3(-1)$ and $K$ are locally free sheaves, the first and the third sheaves in the previous sequence vanish for $p\ge2$, and the first item follows.

The second and the third items are proved in a similar way; compare with \cite[Corollary 7]{J-i}.
\end{proof}

Recall that the singular locus ${\rm Sing}(G)$ of a coherent sheaf $G$ on a nonsingular projective variety $X$ is given by
$$  {\rm Sing}(G) := \{ x\in X ~|~ G_x ~~\text{is not free over}~~ \mathcal{O}_{X,x} \} , $$
where $G_x$ denotes the stalk of $G$ at a point $x$ and $\mathcal{O}_{X,x}$ is its local ring.

It is not difficult to see that the singular locus of a non locally free instanton sheaf $F$ on $\p3$ is precisely the set
$$ {\rm Sing}(F) := \{ x\in \p3 ~|~ \alpha(x) ~~\text{is not injective} \} , $$
where $\alpha(x)$ denotes the map of fibers over the point $x\in \p3$.

Now we focus on rank $2$ instanton sheaves, which are the main object of study in this paper. The following result is proved in \cite[Main Theorem]{JG}.

\begin{theorem}\label{jg-thm}
If $F$ is a non locally free instanton sheaf of rank $2$ on $\p3$, then
\begin{itemize}
\item[(i)] its singular locus has pure dimension $1$;
\item[(ii)] $F^{\dual\dual}$ is a (possibly trivial) locally free instanton sheaf.
\end{itemize}
\end{theorem}

In particular, every relfexive instanton sheaf of rank $2$ is actually locally free, as it can also be deduced from \cite[Proposition 2.6]{H-refl}. We remark that both claims are false for instanton sheaves of rank at least $3$, see \cite{JG} for precise examples.

\begin{lemma}\label{simple h0}
Every nontrivial instanton sheaf $F$ of rank $2$ on $\p3$ is simple and $H^0(F)=0$. If $F$ is locally free, then it is $\mu$-stable. If $F$ is torsion free, then it is $\mu$-semistable.
\end{lemma}

We will see below in Lemma \ref{mu-semistable} that there are non locally free rank $2$ instanton sheaves that are not $\mu$-stable, though certain non locally free rank $2$ instanton sheaves are $\mu$-stable, see Lemma \ref{mu-semistable} below.

\begin{proof}
The simplicity claim is \cite[Lemma 23]{J-i}; the vanishing of $h^0(F)$ is \cite[Proposition 11]{J-i}. If $F$ is locally free, the vanishing of $h^0(F)$ is sufficient to guarantee the $\mu$-stability. From the sequence
$$ 0 \to F \to F^{\dual\dual} \to F^{\dual\dual}/F \to 0 $$
we see that if $F$ is not $\mu$-semistable, then neither is $F^{\dual\dual}$, which is a contradiction.
\end{proof}

\section{Elementary transformations of instantons}\label{degen}

Let $\Sigma$ be a reduced locally complete intersection curve of arithmetic genus $g$. By the degree
of a line bundle $L$ on $\Sigma$ we understand the integer $\deg(L)=\chi(L)+g-1$. For $k\in\mathbb{Z}$ set $\mathrm{Pic}^k(\Sigma):=\{[L]\in\mathrm{Pic}(\Sigma)|\deg(L)=k\}$.

Let $E$ be an instanton sheaf (possibly trivial); \emph{elementary transformation data} $(\Sigma,L,\varphi)$ for $E$ consist of the following:

\begin{itemize}
\item[(i)] an embedding $\iota:\Sigma\hookrightarrow\p3$ of degree $d$;
\item[(ii)] a line bundle $L\in{\rm Pic}^{g-1}(\Sigma)$ such that
$h^0(\iota_*L)=h^1(\iota_*L)=0$;
\item[(iii)] a surjective morphism $\varphi:E\to(\iota_*L)(2)$.
\end{itemize}

\begin{proposition}\label{degdata}
Given an instanton sheaf E of rank $r$ and charge $n$,
and elementary transformation data $(\Sigma,L,\varphi)$ for
$E$ as above, the sheaf $F:=\ker\varphi$ is an
instanton sheaf of rank $r$ and charge $n+d$. Moreover,
$F^{\dual}\simeq E^{\dual}$ and, if $E$ is locally free, then
$F^{\dual\dual}/F\simeq(\iota_*L)(2)$.
\end{proposition}

The sheaf $F:=\ker\varphi$ is called an \emph{elementary transform} of $E$ along $\Sigma$. A similar construction was proposed by Maruyama and Trautmann in \cite[Definition 1.7]{MaTr2} for the case when $\Sigma$ is a union of finitely many disjoint lines.

\begin{proof}
First, applying Riemann--Roch for the immersion $\iota:\Sigma\hookrightarrow\p3$, we have
\begin{equation} \label{rri}
{\rm ch}((\iota_*L)(2)) = \iota_*\left({\rm ch}(L(\,2d\,\pt))\cdot{\rm td}(N_{\Sigma/\p3})^{-1}\right) ,
\end{equation}
where $L\in{\rm Pic}^{g-1}(\Sigma)$ and $N_{\Sigma/\p3}$ is the normal bundle, related to the tangent bundles $T_{\Sigma}$ and $T_{\p3}$ to $\Sigma$ and $\p3$ respectively by the exact triple
\begin{equation} \label{tan-sqc}
0 \to T_{\Sigma} \to \iota^*T_{\p3} \to N_{\Sigma/\p3} \to 0 .
\end{equation}

It follows from (\ref{tan-sqc}) that $c_1(N_{\Sigma/\p3})=4d+2g-2$. Plugging this back into (\ref{rri}), we get
$$
{\rm ch}_3((\iota_*L)(2))=c_1(L(\,2d\,\pt))-c_1(N)/2=
g-1+2d-\frac{1}{2}(4d+2g-2)=0.
$$

Now consider the short exact sequence
\begin{equation}\label{ses-deformation}
0 \to F \stackrel{\delta}{\longrightarrow} E \stackrel{\varphi}{\longrightarrow} (\iota_*L)(2) \to 0 .
\end{equation}
Since ${\rm ch}_1((\iota_*L)(2))=
{\rm ch}_3((\iota_*L)(2))=0$, we obtain $c_1(F)=c_3(F)=0$; since also
${\rm ch}_2((\iota_*L)(2))=[\Sigma]$, we conclude that $c_2(F)=c_2(E)+d=n+d$.

Next, twisting (\ref{ses-deformation}) by $\op3(-2)$ and passing to cohomology, we see that
$h^1(F(-2))=h^2(F(-2))=0$, since by hypothesis $h^p(\iota_*L)=0$ for $p=0,1$.

It is also easy to check that $h^0(E(-1))=h^3(E(-3))=0$ forces $h^0(F(-1))=h^3(F(-3))=0$, since $\Sigma$ is a curve. It then follows that $F$ is an instanton sheaf of rank $r$ and charge $n+d$.

Finally, note that $\inext^1((\iota_*L)(2),\op3)=0$ because the sheaf $(\iota_*L)(2)$ is supported in dimension $1$. Therefore, dualizing the sequence (\ref{ses-deformation}), we obtain the isomorphism $F^\dual\simeq E^\dual$.

For the last claim, consider the diagram
$$ \xymatrix{
& 0 \ar[d] & 0 \ar[d] & & \\
0 \ar[r] & F \ar[d] \ar[r] & E \ar[d]^\simeq \ar[r] & (\iota_*L)(2) \ar[r] & 0 \\
& F^{\dual\dual} \ar[d] \ar[r]^{\simeq} & E^{\dual\dual} & & \\
& F^{\dual\dual}/F \ar[d] & & & \\
& 0 & & &
} $$
This diagram and the Snake Lemma yield the desired isomorphism $F^{\dual\dual}/F\simeq(\iota_*L)(2)$.
\end{proof}

Note that two elementary transforms $F$ and $F'$ of the same instanton sheaf $E$ given by elementary
transformation data $(\Sigma,L,\varphi)$ and $(\Sigma',L',\varphi')$ respectively are isomorphic if and
only if there is an isomorphism $\psi: \iota_*L \to \iota'_*L'$ such that $\psi\circ\varphi=\varphi'$.

Consider now the case when $\Sigma=\Sigma_1\cup\Sigma_2$, where $\Sigma_1$ and $\Sigma_2$ are reduced, locally complete intersection curves such that $\Sigma_1\cup\Sigma_2$ is again a locally complete intersection curve. Take an instanton sheaf $E$ and elementary transformation data
$(\Sigma_1,L_1,\varphi_1)$ for $E$; let $F_1$ be the elementary transform of $E$ along these data.
Then take some elementary transformation data $(\Sigma_2,L_2,\varphi_2)$ for $F_1$,
and let $F_2$ be the corresponding elementary transform. We obtain the diagram
\begin{equation}\label{concatenation}\begin{gathered}
\xymatrix{
& 0 \ar[d] & & & \\
0 \ar[r] & F_2 \ar[d]^{j}\ar[r]^{i\circ j} & E \ar[d]^{\simeq} \ar[r] & Q \ar[r]\ar@{-->}[d] & 0 \\
0 \ar[r] & F_1 \ar[d]^{\varphi_2}\ar[r]^{i} & E \ar[r]^-{\varphi_1} & (\iota_{1*}L_1)(2) \ar[r] & 0 \\
& (\iota_{2*}L_2)(2) \ar[d] & & & \\
& 0 & & &
} \end{gathered}\end{equation}
Thus $F_2$ is an elementary transform of $E$ along a torsion sheaf $Q$ fitting into the short exact sequence 
$$ 0 \to (\iota_{2*}L_2)(2) \to Q \to (\iota_{1*}L_1)(2) \to 0 . $$
Note that if $\Sigma_1$ and $\Sigma_2$ do not intersect, then
$Q=(\iota_{1*}L_1)(2)\oplus (\iota_{2*}L_2)(2)$. We will say that $F_2$ is obtained from
$E$ by a concatenation of the elementary transformations with elementary transformation data
$(\Sigma_1,L_1,\varphi_1)$, $(\Sigma_2,L_2,\varphi_2)$.

\begin{remark}\label{def-transform} \rm
More generally, fix an instanton sheaf $E$ of rank $r$ and charge $n$, and let ${\rm Quot}^{d(k+2)}(E)$ denote the Quot scheme of quotients $\varphi:E\twoheadrightarrow Q$ with Hilbert polynomial $P_Q(k)=d(k+2)$. 

Take $(Q,\varphi)\in{\rm Quot}^{dk+2d}(E)$ satisfying $h^0(Q(-2))=h^1(Q(-2))=0$. Then one can show that the sheaf $F:=\ker\varphi$ is an instanton sheaf of rank $r$ and charge $n+d$. We say that $F$ is a \emph{transform of $E$ along $Q$}.

In this case, one has the short exact sequence
$$ 0 \to F \longrightarrow E \stackrel{\varphi}{\longrightarrow} Q \to 0. $$
Since $E$ is an instanton sheaf and $Q$ is supported in dimension $1$, we have $h^0(F(-1))=h^3(F(-3))=0$. Moreover, the vanishing of $h^0(Q(-2))$ and $h^1(Q(-2))=0$ implies that $h^1(F(-2))=h^2(F(-2))=0$, hence $F$ is a linear sheaf. The Hilbert polynomial of $F$ is given by
$$ P_F(k) = P_E(k) - (dk + 2d) = \frac{r}{6}k^3 + rk^2 + \left(\frac{11}{6}r - (n+d) \right)k + \left(r-2(n+d)\right) , $$
so that $c_1(F)=0$ and $c_2(F)=n+d$.
\end{remark}

\section{Stability of transforms}\label{stab-elm}

In this section we will prove several facts regarding the stability and $\mu$-stability of non locally free instanton sheaves, obtained as transforms of locally free instantons in the way described in Remark \ref{def-transform}.

\begin{lemma}\label{mu-semistable}
If $E$ is a $\mu$-(semi)stable instanton sheaf, then every transform of $E$ is also $\mu$-(semi)stable. Moreover, every transform of the trivial sheaf is $\mu$-semistable but not $\mu$-stable.
\end{lemma}

\begin{proof}
If $F$ is a transform of a $\mu$-(semi)stable instanton sheaf $E$ which is not $\mu$-(semi)stable, then it admits a subsheaf $G$ with $\mu(G)\ge0$ ($\mu(G)>0$). It follows that $G$ would also destabilize $E$.

It follows in particular that every transform $F$ of the trivial sheaf is $\mu$-semistable; to see that it is not $\mu$-stable, just note that $H^0(F^\dual)=H^0(\op3^{\oplus r})\ne0$.
\end{proof}

As is well known, every nontrivial locally free instanton sheaf of rank $2$ is $\mu$-stable. This implies:

\begin{corollary}\label{rk2-mu-stable}
Every transform of a nontrivial locally free instanton sheaf of rank $2$ is $\mu$-stable, and hence stable.
\end{corollary}

Now we will focus on the case of rank $2$. We will see that certain transforms of the trivial rank 2 sheaf are also stable.

In what follows, we denote by $p_F(k)$ the reduced Hilbert polynomial of a sheaf $F$ on $\p3$.

\begin{lemma}\label{stable1}
Every elementary transform of the trivial sheaf of rank $2$ along an irreducible curve is stable.
\end{lemma}

\begin{proof}
Let $F$ be defined by the short exact sequence
$$ 0 \to F \to \op3^{\oplus 2} \xrightarrow{\varphi} Q \to 0  $$
for some elementary transformation data $(\Sigma,L,\varphi)$ with irreducible $\Sigma$, where $Q=(\iota_*L)(2)$ and $\iota:\Sigma\hookrightarrow \p3$ is the
natural embedding.
It is enough to consider torsion free sheaves $G\subset F$ of rank $1$ whose quotient $T:= F/G$ is also torsion free. Moreover, since $F$ is $\mu$-semistable, we can take $\deg(G)=0$, thus $G$ is the ideal sheaf $I_\Delta$ of a subscheme $\Delta\subset\p3$ of dimension at most $1$.

We thus obtain the diagram
\begin{equation}\label{hook-diag}\begin{gathered}
 \xymatrix{
& 0 \ar[d] & 0 \ar[d] & & \\
0 \ar[r] & I_\Delta \ar[d] \ar[r] & \op3 \ar[d]\ar[r] & \imath_{\Delta *}{\mathcal O}_\Delta \ar[r]\ar@{-->}[d]^{\zeta} & 0 \\
0 \ar[r] & F \ar[d] \ar[r] & \op3^{\oplus 2} \ar[d]\ar[r] & Q \ar[r] & 0 \\
& T \ar[d] & \op3 \ar[d] & & \\
& 0 & 0 & &
} \end{gathered}\ \ \ ,\end{equation}
where $\imath_{\Delta}:\Delta\hookrightarrow\p3$ is the natural embedding.
It provides a map $\zeta:{\mathcal O}_\Delta\to Q$ and an exact sequence
$$ 0 \to \ker\zeta \to T \to \op3 \to \coker\zeta \to 0 .$$
Since $T$ is torsion free, it follows that $\ker\zeta=0$ and ${\mathcal O}_\Delta$ is a subsheaf of $Q$.
Since $\Sigma$ is irreducible, we must then have that either $\Delta=\Sigma$, and hence $G=I_\Sigma$, or $\Delta=\varnothing$, and hence $G=\op3$.

Since $H^0(F)=0$ by Lemma \ref{simple h0}, the second case does not occur. In the first case we have,
$$ p_F(k) - p_{I_\Sigma}(k) = \frac{d}{2} k > 0 , $$
where $d:=\deg(\Sigma)=c_2(F)$. Hence $F$ is stable.
\end{proof}

\begin{remark} \rm
Note that we do not use the hypothesis $h^0(Q(-2))=h^1(Q(-2))=0$ in the proof above. In other words, if $Q$ is sheaf with Hilbert polynomial $P_Q(k)=dk+2d$ and irreducible support, then the kernel of a surjective morphism $\op3^{\oplus2}\twoheadrightarrow Q$ is a stable rank $2$ torsion free sheaf $F$ with $c_1(F)=c_3(F)=0$ and $c_2(F)=d$.
\end{remark}

As a by-product of the previous proof we also obtain the following interesting fact.

\begin{corollary}
Every elementary transform of $\op3^{\oplus2}$ along an irreducible curve $\Sigma$ of genus $g$ and degree $d$ is an extension of an ideal sheaf $I_Z$ of a $0$-dimensional subscheme $Z\subset\p3$ of length $2d+g-1$ by the ideal sheaf
$I_\Sigma$.
\end{corollary}

We now consider another situation, given by a concatenation of two elementary transformations along irreducible rational curves, that will be relevant later on.

Let $\imath:\ell\hookrightarrow\p3$ be a line and let $\jmath:\Gamma\hookrightarrow\p3$ be a rational curve of degree $m-1$ ($m\ge 2$); assume that either $\ell\cap \Gamma=\varnothing$, or $\ell$, $\Gamma$ intersect quasi-transversely at a single point $P$; we say that the intersection is quasi-transverse when the tangents to $\Gamma,\ell$ at the intersection point are distinct. Consider the sheaf $Q$ given by an extension
\begin{equation}\label{defn-Q}
0 \to (\imath_*\ol(-\pt))(2) \longrightarrow Q
\stackrel{\tau}{\longrightarrow} (\jmath_*\mathcal{O}_\Gamma(-\pt))(2) \to 0 .
\end{equation}
Such an extension may be nontrivial only if $\Gamma,\ell$ intersect. So, either $$Q\simeq
(\imath_*\ol(\pt))\oplus (\jmath_*\mathcal{O}_\Gamma((2m-1)\pt)),$$ or the curve $\Gamma\cup\ell$ is connected
and $Q$ is a line bundle on $\Gamma\cup\ell$ such that $$Q|_{\ell}\simeq\calo_{\ell}(2\pt),\ \ 
Q|_{\Gamma}\simeq\calo_{\ell}((2m-1)\pt).$$

Let $E$ be an instanton sheaf of rank $r$ and charge $c$, and assume there exists a surjective map $\varphi:E\to Q$. Let $F:=\ker\varphi$. We have the following diagram:
\begin{equation}\label{concat2}\begin{gathered}
\xymatrix{
& & & 0 \ar[d] & \\
& & & (\imath_*\ol(-\pt))(2) \ar[d] & \\
0 \ar[r] & F \ar[r] & E \ar[d]^{\simeq} \ar[r]^{\varphi} & Q \ar[r]\ar[d]^{\tau} & 0 \\
0 \ar[r] & F' \ar[r] & E \ar[r]^-{\tau\circ\varphi} & (\jmath_*\mathcal{O}_\Gamma(-\pt))(2) \ar[r]\ar[d] & 0 \\
& & & 0 &
} \end{gathered}\ \ \ \ ,\end{equation}
where $F':=\ker\tau\circ\varphi$. We then obtain the short exact sequence
\begin{equation} \label{ef sqc}
0 \to F \to F' \to (\imath_*\ol(-\pt))(2) \to 0 .
\end{equation}
Comparing with diagram (\ref{concatenation}), we see that $F$ is obtained by concatenation of two elementary transformations, first along $\Gamma$ then along $\ell$.

\begin{lemma}\label{stable2}
Every transform of the trivial sheaf of rank $2$ along a sheaf $Q$ given by an extension (\ref{defn-Q}) is stable.
\end{lemma}
\begin{proof}
Let $F$ be defined by the exact triple
$$ 0 \to F \to \op3^{\oplus 2} \to Q \to 0 . $$
As observed above, $F$ also fits into the short exact sequence (\ref{ef sqc}), where $F'$ is an elementary transform of the trivial sheaf of rank $2$ along $\Gamma$.

Proceeding as in the beginning of the proof of Lemma \ref{stable1}, we conclude that any rank $1$ torsion free subsheaf $G\subset F$ with torsion free quotient $T:=F/G$ and zero degree is the ideal sheaf $I_\Delta$ of a closed subscheme $\Delta\subset\p3$ of dimension at most $1$, and that $\mathcal O_\Delta$ is a subsheaf of $Q$.
There are only four possibilities for $\Delta$: $\varnothing$, $\ell$, $\Gamma$ and $\Gamma\cup\ell$; then, respectively, $G=\op3$, $G=I_\ell$, $G=I_\Gamma$, and $G=I_{\Gamma\cup\ell}$.

The first possibility does not occur, since $H^0(F)=0$. Let us examine the other three possibilities.

In the case  $G=I_\ell$, we have $G\subset F\subset F'$ and
$$ p_{F'}(k) - p_{I_\ell}(k) = \frac{(3-m)}{2}k + 2 - m . $$
If $m\ge3$, we get $ p_{F'}(k)-p_{I_\ell}(k)<0$, so $G$ destabilizes $F'$, which is impossible by Lemma \ref{stable1}. Assume now that $G=I_\ell$ and $m=2$. Then $\ell$, $\Gamma$ are two lines spanning a plane, and $I_\ell$ does not destabilize $F'$, but $ p_{F}(k)-p_{I_\ell}(k)=-1$ and so $I_\ell$ destabilizes $F$. Let us see that this case is impossible.

Consider the diagram \eqref{hook-diag} for our sheaf $F$ with $\Delta=\ell$. Arguing as in the proof of Lemma \ref{stable1}, we can complete it to the diagram
\begin{equation}\label{3x3-diag}\begin{gathered}
 \xymatrix{
& 0 \ar[d] & 0 \ar[d] & 0 \ar[d] & \\
0 \ar[r] & I_\ell \ar[d] \ar[r] & \op3 \ar[d]\ar[r] & \iota_*{\mathcal O}_\ell \ar[r]\ar[d]^{\zeta} & 0 \\
0 \ar[r] & F \ar[d] \ar[r] & \op3^{\oplus 2} \ar[d]\ar[r] & Q  \ar[d]\ar[r] & 0 \\
0 \ar[r] & T \ar[d]\ar[r] & \op3\ar[r] \ar[d] & \coker\zeta\ar[r] \ar[d] & 0\\
& 0 & 0 & 0 &
} \end{gathered}\end{equation}

If, for instance, $\Gamma$, $\ell$ intersect and the extension \eqref{defn-Q} is non-trivial, then $Q|_{\ell}\simeq(\iota_*{\mathcal O}_\ell)(2)$
and $Q|_{\Gamma}\simeq(\jmath_*{\mathcal O}_\Gamma)(1)$. As $\zeta$ factors
through $Q|_{\ell}(-\pt)\hookrightarrow Q$,
$$
\coker\zeta\simeq (\jmath_*{\mathcal O}_\Gamma)(1)\oplus \C_{P'},
$$
where $\C_{P'}$ denotes a sky-scraper sheaf of length 1 supported at a point $P'\in\ell$.
It is obvious that $\coker\zeta$ cannot be generated by a single section, so that the
surjection $\op3\to \coker\zeta$ in the last line of \eqref{3x3-diag} does not exist.
By a completely similar argument, one treats the remaining cases: (a)~$\Gamma$, $\ell$ intersect, 
but the extension \eqref{defn-Q} is trivial, and (b)~$\Gamma$, $\ell$ do not intersect.
The conclusion is that $G=I_\ell$ for $m=2$ cannot occur as a destabilizing subsheaf.

Next, the ideal sheaf $I_\Gamma$ does not destabilize $F$ when $m\ge3$, for
$$ p_F(k) - p_{I_\Gamma}(k) = \frac{m-2}{2}k + 1 - m > 0 , $$
and it does not destabilize $F$ for $m=2$ by an analysis of a commutative diagram,
similar to \eqref{3x3-diag} (see also Remark \ref{permute-ell-gamma} below).

Finally, we check that the ideal sheaf $I_{\Gamma\cup\ell}$ also does not destabilize $F$ when $m\ge2$; indeed:
$$ p_E(k) - p_{I_{\Gamma\cup\ell}}(k) = \frac{m}{2}k + (2 - m) > 0 . $$
\end{proof}

\begin{remark}\label{permute-ell-gamma}
{\em
Observe that we can permute the roles of $\ell,\Gamma$ in the construction
of the concatenation. In the
same notation as in the paragraph preceding equation (\ref{defn-Q}),
let $Q$ be given by a nontrivial extension
$$
0 \to (\jmath_*\mathcal{O}_\Gamma(-\pt))(2) \longrightarrow Q
\stackrel{\tau}{\longrightarrow} (\imath_*\ol(-\pt))(2) \to 0 ,
$$
and let $F$ be the rank 2 instanton sheaf obtained by an elementary
transformation of the trivial sheaf of rank $2$ along $Q$:
$$ 0 \to F \to \op3^{\oplus 2} \to Q \to 0 . $$
We then have the following two exact triples:
$$ 0 \to F \to F' \to (\jmath_*\mathcal{O}_\Gamma(-\pt))(2) \to 0 ~~{\rm
and}~~
0 \to F' \to \op3^{\oplus 2} \to (\imath_*\ol(-\pt))(2) \to 0. $$
Note that $F'$ is stable by Lemma \ref{stable1}, and the stability of
$F$ is proved by an argument similar to that of Lemma \ref{stable2}.

For example, the ideal sheaf $I_{\ell}$ would destabilize $F$ for some $m\ge 2$, 
if one might find an embedding $I_\ell\hookrightarrow F$ with torsion free
quotient $T=F/I_\ell$. Let us assume that such an embedding exists.
Then we can complete it to a diagram of the form \eqref{3x3-diag}. We observe that
$Q|_{\ell}\simeq(\iota_*{\mathcal O}_\ell)(1)$
and $Q|_{\Gamma}\simeq\jmath_*{\mathcal O}_\Gamma((2m-2)\pt)$, hence
$\coker\zeta\simeq\jmath_*{\mathcal O}_\Gamma((2m-2)\pt)$ cannot be generated by a single
section, and the diagram \eqref{3x3-diag} does not exist by the same reason as in the proof
of Lemma \ref{stable2}.
}
\end{remark}

\section{Hom's and Ext's of elementary transforms}

We start by the following general claim.

\begin{lemma}\label{quot-homs}
Let $\iota:\Sigma\hookrightarrow\p3$ be a reduced locally complete intresection curve, $M$ an invertible
$\osigma$-sheaf, and $E$ a rank 2 locally free sheaf on $\p3$ equipped with a surjective map $\varphi:E\to\iota_*M$.
For the torsion free sheaf $F:=\ker\varphi$, we have
$\inhom(E,E)/\inhom(F,F) \simeq \iota_*M^2\otimes\det(E)^\dual$.
\end{lemma}

\begin{proof}
First, apply $\inhom(-,E)$ to the short exact sequence
\begin{equation}\label{claim-sqc}
0 \to F \to E \to \iota_*M \to 0 .
\end{equation}
Since $\iota_*M$ is a torsion sheaf supported on a curve and $E$ is locally free, we have $\inhom(\iota_*M,E)=\inext^1(\iota_*M,E)=0$, and conclude that $\inhom(E,E)\simeq\inhom(F,E)$.

Next, applying $\inhom(\iota_*M,-)$ to the sequence (\ref{claim-sqc}) and again using the vanishing of $\inhom(\iota_*M,E)$ and $\inext^1(\iota_*M,E)$, we obtain that
$\inext^1(\iota_*M,F)\simeq\inhom(\iota_*M,\iota_*M)=\iota_*\osigma$.

Now, the diagram
\begin{equation}\label{big-diagram}\begin{gathered}
\xymatrix {
         & 0 \ar[d]                          &                                    &                           &   \\
0 \ar[r] & \inhom(E,F) \ar[d] \ar[r]         & \inhom(E,E) \ar[r] \ar[d]^{\simeq} & \inhom(E,\iota_*M) \ar[r] & 0 \\
0 \ar[r] & \inhom(F,F) \ar[d] \ar[r]^{\tau}  & \inhom(E,E) \ar[r]                 & \coker\tau \ar[r]         & 0 \\
         & \iota_*\osigma \ar[d]             &                                    &                           &   \\
         & 0                                 &                                    &                           &  }\end{gathered}
\end{equation}
is obtained in the following way. The first row comes from applying $\inhom(E,-)$ to the sequence (\ref{claim-sqc}). The second row comes from applying $\inhom(F,-)$ to the same sequence, and using the identification $\inhom(E,E)\simeq\inhom(F,E)$. The left column comes from applying $\inhom(-,F)$ to (\ref{claim-sqc}), noting that $\inext^1(\iota_*M,F)\simeq\iota_*\osigma$ and $\inext^1(\iota_*M,E)=0$.

Now the Snake Lemma provides us with the short exact sequence
$$ 0 \to \iota_*\osigma \to E^\dual\otimes\iota_*M \to \coker\tau \to 0 ,$$
since $\inhom(E,\iota_*M)\simeq E^\dual\otimes\iota_*M$. Since $E$ has rank $2$, it easily follows that \linebreak
$\inhom(E,E)/\inhom(F,F) = \coker\tau \simeq \iota_*(M^2)\otimes\det(E)^\dual$, as desired.
\end{proof}

Next, we apply the previous Lemma to the case of elementary transformations of locally free instanton sheaves.

\begin{lemma}\label{global exts}
Let $E$ be a rank $2$ locally free instanton sheaf of charge $n$, and let $(\Sigma,L,\varphi)$ be elementary transformation data for $E$. If $L$ is an invertible $\osigma$-sheaf satisfying $h^1((\iota_*L^2)(4))=0$, then for the sheaf $F=:\ker\varphi$, we have
\begin{itemize}
\item[(i)] $\ext^1(F,F) = H^0(\inext^1(F,F))\oplus H^1(\inhom(F,F))$;
\item[(ii)] $\ext^2(F,F) = H^1(\inext^1(F,F)) \simeq H^1(\inext^2((\iota_*L)(2),F))$;
\end{itemize}
Furthermore, one has
\begin{equation}\label{inhom dim}
h^1(\inhom(F,F)) = 8n + h^0((\iota_*L^2)(4)) - 3
\end{equation}
\end{lemma}

\begin{proof}
We use the local-to-global spectral sequence for Ext's, whose $E_2$ term is of the form
$$ E^{pq}_2 = H^p(\inext^q(F,F)) . $$

First, from Lemma (\ref{exts}(i)) we know that $\inext^q(F,F)=0$ for $q=2,3$, killing the terms $E^{pq}_2$ for $q=2,3$. Moreover, since $\inext^1(F,F)$ is suported on $\Sigma$, it follows that $E^{p1}_2$ for $p=2,3$.

Applying Lemma \ref{quot-homs} to the sequence
$$ 0 \to F(-2) \to E(-2) \to \iota_*L \to 0 ,$$
we obtain the short exact sequence
\begin{equation}\label{sqc inhom}
0 \to \inhom(F,F) \to \inhom(E,E) \to (\iota_*L^2)(4) \to 0.
\end{equation}
Since $\iota_*L^2$ is supported in dimension $1$, passing to cohomology, we conclude that
$H^3(\inhom(F,F))\simeq Ext^3(E,E)=0$ by (\ref{exts} (iii)); thus $E^{03}_2$ also vanishes.

Finally, sequence (\ref{sqc inhom}) and the hypothesis $h^1((\iota_*L^2)(4))=0$ imply that \linebreak
$H^2(\inhom(F,F))\simeq H^2(\inhom(E,E))$. The last cohomology vanishes by hypothesis, thus also $E^{02}_2=0$.

It then follows that the spectral sequence already degenerates at the second page, and one concludes that:
$$ \ext^1(F,F) = H^0(\inext^1(F,F)) \oplus H^1(\inhom(F,F)) ~~{\rm and}~~
\ext^2(F,F) = H^1(\inext^1(F,F)), $$
as desired. The isomorphism
$H^1(\inext^1(F,F)) \simeq H^1(\inext^2((\iota_*L)(2),F))$ is obtained from applying the functor
$\inhom(-,F)$ to the sequence (\ref{ses-deformation}).

In order to establish formula (\ref{inhom dim}), consider the cohomology exact sequence associated to the exact triple of sheaves (\ref{sqc inhom}):
\begin{multline*}
0 \to H^0(\inhom(F,F))   \to H^0(\inhom(E,E)) \to H^0((\iota_*L^2)(4)) \to
 \\  H^1(\inhom(F,F)) \to H^1(\inhom(E,E)) \to 0.
\end{multline*}
Since every nontrivial rank $2$ instanton sheaf is simple, we have \linebreak $h^0(\inhom(F,F))=1$. Then, counting dimensions in the sequence above, we have
$$ h^1(\inhom(F,F)) = h^1(\inhom(E,E)) + h^0((\iota_*L^2)(4)) + 1 - h^0(\inhom(E,E)). $$

There are now two cases to consider. First, if $E$ is a nontrivial rank 2 locally free instanton of charge $n$, it follows that $\mathrm{Ext}^2(E,E)=0$ (see \cite{JV}). Moreover, $h^0(\inhom(E,E))=1$, since $E$ is simple. Therefore, $h^1(\inhom(E,E))=8n-3$, hence the desired formula follows.

On the other hand, if $E=\op3^{\oplus 2}$ (i.e. if $n=0$), then $h^0(\inhom(E,E))=4$ and $h^1(\inhom(E,E))=0$, hence one also obtains formula (\ref{inhom dim}).
\end{proof}

Lemma \ref{global exts} prompts us to characterize the sheaf $\inext^1(F,F)$ when $F$ is an elementary transform of a rank $2$ locally free instanton sheaf. Applying the functor
$\inhom(F,-)$ to sequence (\ref{ses-deformation}) we obtain
\begin{multline}\label{inhom sqc 1}
0 \to \inhom(F,F) \to \inhom(F,E) \to \inhom(F,(\iota_*L)(2)) 
\to \\ \inext^1(F,F) \to \inext^1(F,E) \to \inext^1(F,(\iota_*L)(2)) \to 0 , 
\end{multline}
since $\inext^2(F,F)=0$ by Lemma \ref{exts}. Invoking now Lemma \ref{quot-homs} and the isomorphism
$\inhom(F,E)\simeq\inhom(E,E)$ (see the first paragraph of the proof of Lemma \ref{quot-homs}) sequence (\ref{inhom sqc 1}) can be rewritten in the following way:
\begin{multline}\label{inhom sqc 2}
0 \to (\iota_*L^2)(4) \to \inhom(F,(\iota_*L)(2))  
\to \inext^1(F,F) \to\\ \inext^1(F,\op3)\otimes E \to \inext^1(F,(\iota_*L)(2)) \to 0 . 
\end{multline}

In the next section, we will carefully analyze each term of this exact sequence for elementary transforms along rational curves.

\section{Definition and properties of $\overline{\cald(m,n)}$} \label{ratcurves}

Let $\calr_{0}(m)$ denote the space of nonsingular rational curves of degree $m$ on $\p3$. It is well-known (see e.~g. \cite{EV}) that $\calr_{0}(m)$ is
a nonsingular irreducible quasiprojective variety of dimension $4m$ and that the set
$$
\calr^*_{0}(m):=\{\Gamma\in \calr_{0}(m) ~|~ 
N_{\Gamma/\p3}\simeq\mathcal{O}_\Gamma((2m-1)\pt )^{\oplus 2}\},
$$
where $\pt$ denotes a point of $\Gamma$, is a dense open subset of $\calr_{0}(m)$ for $m\ge1,\ m\ne2$. Besides, it is obvious that
$N_{\Gamma/\p3}\simeq\mathcal{O}_\Gamma(2\pt )\oplus\mathcal{O}_\Gamma(4\pt )$ for every
$\Gamma\in\calr_{0}(2)$. We set $\calr^*_{0}(2):=\calr_{0}(2)$.

\begin{lemma}\label{res-trivial}
Let $E$ be an instanton sheaf. Then for any $m\geq 1$, the restriction of $E$ to a generic
rational curve of degree $m$ in $\mathbb P^3$ is trivial.
\end{lemma}

\begin{proof}
For $m=1$, the assertion follows from the $\mu$-semistability of $E$ and the Grauert--Müllich
Theorem \cite[Theorem 3.1.2]{HL}. For $m>1$, we start by restriction to a generic chain of $m$ lines and then smooth out the chain of lines to a nonsingular rational curve of degree $m$.

By a chain of lines we mean a curve
$\Gamma_0=\ell_1\cup... \cup \ell_m$ in $\mathbb{P}^3$ such that $\ell_1,...,\ell_m$ are distinct lines and $\ell_i\cap\ell_j\neq\varnothing$ if and only if $|i-j|\leq 1$. It is well known (see e.~g. \cite[Corollary 1.2]{HH2}) that a chain of lines
$\Gamma_0=\ell_1\cup... \cup \ell_m$ in $\mathbb{P}^3$ considered as a reducible curve of degree $m$ can be deformed in a flat family with a smooth one-dimensional base $(\Delta, 0)$ to a nonsingular rational curve $\Gamma\in \calr_{0}(m)$. Making an étale base change, we can obtain such a smoothing with a cross-section.

By the case $m=1$, the restriction of $E$ to a generic line is trivial. By induction on $m$, we easily deduce that for a generic chain of lines $\Gamma_0$, the restriction of $E$
to $\Gamma_0$ is also trivial: $E|_{\Gamma_0}\simeq\mathcal{O}_{\Gamma_0}^{\oplus2}$, which is equivalent to saying that $E|_{\ell_i}\simeq\mathcal{O}_{\ell_i}^{\oplus2}$ for all
$i=1,\ldots,m$. Choosing a smoothing $\{\Gamma_t\}_{t\in\Delta}$ of $\Gamma_0$
with a cross-section $t\mapsto x_t\in\Gamma_t$ as above, we remark that
$E|_{\Gamma_t}\simeq \mathcal{O}_{\Gamma_t}(k_t\pt)\oplus\mathcal{O}_{\Gamma_t}(-k_t\pt)$
for some integer $k_t$ which may depend on $t$. The triviality of
$E|_{\Gamma_t}$ is thus equivalent to the vanishing of $h^0(E|_{\Gamma_t}(-\pt))$.
Using the semi-continuity
of $h^0(E|_{\Gamma_t}(-x_t))$, we see that $E|_{\Gamma_t}$ is trivial for generic $t\in\Delta$.
\end{proof}

\begin{lemma}\label{R*0ME}
Let $1\le m\le n$, and let $E$ be an instanton sheaf of charge $n$. Then
$$
\calr^*_{0}(m)_E := \{\:\Gamma\in \calr^*_{0}(m)~|~ E|_\Gamma\simeq\mathcal{O}_\Gamma^{\oplus2}\ \}
$$
is a nonempty open subset of $\calr^*_{0}(m)$. Moreover,
$$
\calb(m,n) := \{\:([E],\Gamma)\in\cali(n-m)\times \calr^*_{0}(m)\ |\
\Gamma\in \calr^*_{0}(m)_E\ \}
$$
is a nonempty open subset of $\cali(n-m)\times \calr^*_{0}(m)$, whose projections to
both factors are surjective.
\end{lemma}

\begin{proof}
To prove the first assertion, note that  $\calr^*_{0}(m)_E$ is nonempty by
Lemma \ref{res-trivial}. It thus suffices to show that any $\Gamma\in
\calr^*_{0}(m)_E$ has a Zariski open neighborhood in $\calr^*_{0}(m)$, entirely contained
in $\calr^*_{0}(m)_E$. 

For any $\Gamma\in
\calr^*_{0}(m)_E$, we can find a plane $\Pi\simeq\p2$
in $\p3$ meeting $\Gamma$ transversely in $m$ distinct points. Then the set $U=U(\Pi)\subset\calr^*_{0}(m)$ of curves $\Gamma'\in \calr^*_{0}(m)$ that meet
$\Pi$ transversely in $m$ points is a Zariski neighborhood of $\Gamma$ as a point of
$\calr^*_{0}(m)$. Define
$$
\tilde U=\tilde U(\Pi):=\{\:(x,\Gamma)\in \Pi\times U\ | \ x\in \Gamma\ \}.
$$
Let $\boldsymbol{\Gamma}\subset\p3\times\calr^*_{0}(m)$ be the universal degree-$m$ rational curve over
$\calr^*_{0}(m)$, and for any morphism $T\to \calr^*_{0}(m)$, denote by
$\boldsymbol{\Gamma}_T$ the pullback $T\times_{\calr^*_{0}(m)}\boldsymbol{\Gamma}
\subset \p3\times T$
of $\boldsymbol{\Gamma}$ via this morphism. For $t\in T$, we will denote by
$\Gamma_t$ the fiber of $\boldsymbol{\Gamma}_T$ over $t$.
The natural map $\tilde U\to U$ is an étale covering of degree $m$, and the pullback
$\boldsymbol{\Gamma}_{\tilde U}\to \tilde U$ admits $m$ disjoint cross-sections
$x_i:\tilde U \to\boldsymbol{\Gamma}_{\tilde U}$ 
($i=1,\ldots, m$) with images in $\Pi\times\tilde U$. For ${\tilde u} \in \tilde U$, the triviality of
$E|_{\Gamma_{\tilde u}}$ is equivalent to the vanishing $h^0(E|_{\Gamma_{\tilde u}}(-x_1({\tilde u})))=0$, so, by
the semi-continuity of $h^0$, the subset $\tilde V\subset\tilde U$ of points ${\tilde u}\in\tilde U$
such that $E|_{\Gamma_{\tilde u}}$ is trivial is open in $\tilde U$. By the openness of finite étale morphisms, the image $V$ of $\tilde V$ in $U$ is open. This is an open neighborhood of
$\Gamma$ contained in $\calr^*_{0}(m)_E$, so the first assertion is proved.

The second assertion is obtained by relativizing the above argument
over \mbox{$\cali(n-m)$.} This is straightforward when
$\cali(n-m)$ possesses a universal instanton bundle. By \cite[Section 4.6]{HL},
this happens when $n-m$ is odd. Moreover, the assertion is trivial for $n-m=0$,
in which case the only instanton bundle is $\op3^{\oplus 2}$. In the case of even $n-m\geq 2$, the argument extends verbatim by replacing ``universal'' by ``quasi-universal''. 
A quasi-universal sheaf $\boldsymbol{\widetilde E}$ of multiplicity
$\nu\geq 1$ is a sheaf over $\p3\times \calm^{\mathrm {st}}(n-m)$
such that $\boldsymbol{\widetilde E}|_{\p3\times\{t\}}\simeq E_t^{\oplus \nu}$ for every $t\in\calm^{\mathrm {st}}(n-m)$, where $\calm^{\mathrm {st}}(n-m)$
is the locus of stable sheaves in $\calm(n-m)$, and $E_t$ denotes a stable sheaf whose isomorphism class is represented by $t$. By loc. cit., there exists
a quasi-universal sheaf of multiplicity $\nu=1$ if $n-m$ is odd (and then it
is indeed universal) and $\nu=2$ if $n-m$ is even.

Taking a pair $(t,\Gamma)\in \calb(m,n)$, we choose a plane $\Pi$ in $\PP^3$ meeting $\Gamma$
transversely and define $U,\tilde U$ as above, and consider the
relative degree-$m$ rational curve
$\Upsilon:=\cali(n-m)\times \boldsymbol{\Gamma}_{\tilde U}$
with a natural projection $\pi:\Upsilon\to \cali(n-m)\times {\tilde U}$. Then the set
$$
{\widetilde V}=
\left\{ (t,\tilde u)\in \cali(n-m)\times {\tilde U}\ \left| \ 
h^0\left(\boldsymbol{\widetilde E}_\Upsilon|_{\pi^{-1}(t,\tilde u)}(-x_1(\tilde u))\right)=0
\right.\right\}
$$
is open by Grauert's semicontinutity. Hence its image ${ V}$ under the étale
map $\cali(n-m)\times {\tilde U}\to \cali(n-m)\times \calr^*_{0}(m)$ is an open
neighborhood of $(t,\Gamma)$, and the openness of $\calb(m,n)$ is proved.

\end{proof}

\begin{corollary} \label{dim B}
$\calb(m,n)$ is a nonsingular irreducible quasiprojective variety of dimension\sloppy\ 
\mbox{$8n-4m-3$}.
\end{corollary}

For $(t,\Gamma)\in \calb(m,n)$, let $L$ be the line bundle $\mathcal{O}_\Gamma(-\pt )$ of degree $-1$ on $\Gamma$. Denoting by $\iota$ the natural embedding
$\Gamma\hookrightarrow\mathbb P^3$, we have $h^0(\iota_*L)=h^1(\iota_*L)=0$ and
$$ \Hom(E_t,(\iota_*L)(2)) \simeq \Hom(\mathcal{O}_\Gamma^{\oplus 2}, \mathcal{O}_\Gamma((2m-1)\pt )) \simeq H^{0}(\mathcal{O}_\Gamma((2m-1)\pt ))^{\oplus 2}. $$
We thus conclude that there exists a surjective map $\varphi:E_t\to(\iota_*L)(2)$, so that  
$(\Gamma,L,\varphi)$ is a set of elementary transformation data for $E_t$.

By Proposition \ref{degdata}, Corollary \ref{rk2-mu-stable} and Lemma \ref{stable1}, 
$F:=\ker\varphi$ is a stable rank $2$ instanton sheaf of charge $n$, whose
double dual is isomorphic to $E_t$. For $n\ge1$ and each $m=1,\dots,n$, consider the subset $\cald(m,n)$ of~$\calm(n)$ consisting of the isomorphism classes $[F]$ of the sheaves $F$
obtained in this way:
$$ \cald(m,n) :=
\{ [F]\in\calm(n) ~|~ [F^{\dual\dual}]\in\cali(n-m), ~~
\Gamma=\mathrm{Supp}(F^{\dual\dual}/F)\in \calr^*_{0}(m)_{F^{\dual\dual}}~,
$$
$$
~~{\rm and}~~
F^{\dual\dual}/F\simeq(\iota_*L)(2), ~ \mathrm{where}  ~ L=\mathcal{O}_\Gamma(-\pt )
\}.
$$
Let $\overline{\cald(m,n)}$ denote the closure of $\cald(m,n)$ in $\calm(n)$.

\begin{proposition}\label{variety D}
For $n\ge1$ and $1\le m\le n$, $\overline{\cald(m,n)}$ is an irreducible projective variety of dimension $8n-4$, and $\cald(m,n)$ is open in $\overline{\cald(m,n)}$.
\end{proposition}

\begin{proof}
We will construct a Severi--Brauer variety $\calp=\calp(m,n)$, fibered over $\calb=\calb(m,n)$ with fiber
$\p{4m-1}$, and an open subset $\calz\subset \calp$ together
with a morphism $\calz\to \calm(n)$ mapping $\calz$ bijectively onto $\cald=\cald(m,n)$.

According to \cite[Section 4.6]{HL}, a universal rank-2 sheaf exists locally
in the étale topology over the stable locus $\calm^{\mathrm {st}}$ of $\calm=\calm(n-m)$. Thus there exists an open étale
covering $\Phi:W\to \calm^{\mathrm {st}}$ and a rank 2 sheaf
$\mathbf E=\mathbf E_W$ over $\PP^3\times W$ such that
for any $w\in W$, $\mathbf E|_{\PP^3\times w}\simeq E_t$, 
where $t=\Phi(w)\in\calm^{\mathrm {st}}$ and $E_t$
denotes the stable sheaf whose isomorphism class is represented by $t$.
Refining the étale covering, one can define 
``transition functions'' $g:\pr_{12}^*\mathbf E\isoto
\pr_{13}^*\mathbf E$ over $\PP^3\times W\times_{\calm^{\mathrm {st}} }W$, where
$\pr_{ij\ldots}$ denotes the projection to
the product of the $i$-th, $j$-th\dots\ factors of a product of several factors.
The transition functions, in general, fail to satisfy the cocycle condition, so
$\mathbf E$ does not descend to $\calm^{\mathrm {st}}$. The failure of $g$ to
satisfy the 1-cocycle condition is measured by a gerbe $c_g\in \Gamma(W\times_{\calm^{\mathrm {st}} } W\times_{\calm^{\mathrm {st}} } W,
\calo^*)$, which is an étale 2-cocycle representing a $\nu$-torsion element $\alpha=[c_g]$ of the Brauer group
$\mathrm{Br}\,(\calm^{\mathrm {st}} )\subset H^2_{\mathrm{\acute et}}(\calm^{\mathrm {st}} , \calo^*)$, where 
$\nu=1$ or 2 depending on the parity of $n-m$.
Over $\PP^3\times W\times_{\calm^{\mathrm {st}} } W\times_{\calm^{\mathrm {st}} } W$, the following twisted cocycle condition holds: 
 \begin{equation}\pr^*_{134}g\circ\pr^*_{123}g=\pr_{234}^*c_g\cdot\pr^*_{124}g.
\label{twisted-cocycle}\end{equation}

We now choose an étale open covering $\boldsymbol{\widetilde V}\to\calb$ by
the étale open sets ${\widetilde V}$ constructed in the end of the proof of
Lemma \ref{R*0ME}, and define $\boldsymbol{\widetilde U}=W\times_{\calm^{\mathrm {st}}} \boldsymbol{\widetilde V}$.
 We also denote by $\widetilde{\boldsymbol{\Gamma}}$ the 
pullback of the universal degree-$m$ rational curve in $\PP^3$
with two natural maps $\tilde \pi:\widetilde{\boldsymbol{\Gamma}}\to
\boldsymbol{\widetilde U}$, $\boldsymbol\iota : \widetilde{\boldsymbol{\Gamma}}
\into \p3\times\boldsymbol{\widetilde U}$, 
and by
$\widetilde{\mathbf E}$ the pullback of $\mathbf E$ to $\widetilde{\boldsymbol{\Gamma}}$.
Then $\tilde \pi$ possesses $m$ disjoint cross-sections over $\boldsymbol{\widetilde U}$,
which we denote $x_1,\ldots, x_m$, as before, and we can choose a global line bundle 
$\mathbf L$ on $\widetilde{\boldsymbol{\Gamma}}$ of relative degree $-1$ over
$\boldsymbol{\widetilde U}$ by setting 
$\mathbf L:=\calo_{\widetilde{\boldsymbol{\Gamma}}}(-x_1)$. For an integer $k$, we denote
by $\calo_{\widetilde{\boldsymbol{\Gamma}}}(k)$ the pullback of $\calo_{\p3}(k)$,
and $\mathbf L(k):=\mathbf L\otimes\calo_{\widetilde{\boldsymbol{\Gamma}}}(k)$.

For any $\tilde u\in \boldsymbol{\widetilde U}$, we have 
$\widetilde{\mathbf E}_{\Gamma_{\tilde u}}
\simeq \calo^{\oplus 2}_{\Gamma_{\tilde u}}$,
and $\Hom(\widetilde{\mathbf E}_{\Gamma_{\tilde u}},\mathbf L(2)|_{\Gamma_{\tilde u}})
\simeq H^0(\Gamma_{\tilde u}, \calo^{\oplus 2}_{\Gamma_{\tilde u}}((2m-1)\pt))$,
where
$\Gamma_{\tilde u}:=\tilde\pi^{-1}(\tilde u)$, is a $4m$-dimensional vector space. 
These vector spaces glue into the vector bundle 
$\boldsymbol\tau:=\HOM_{\widetilde{\boldsymbol{\Gamma}}/\boldsymbol{\widetilde U}}(\mathbf E_{\boldsymbol{\widetilde U}},
\boldsymbol\iota_*\mathbf L(2))$
over~$\boldsymbol{\widetilde U}$.
Its associated projective bundle $\boldsymbol{p_{\widetilde U}}:
\mathbf P_{\boldsymbol{\widetilde U}}:=
\PP\boldsymbol\tau\to \boldsymbol{\widetilde U}$ 
descends to $\calb$, that is, there exists
a Severi--Brauer variety $\boldsymbol{p}_\calb:\mathbf P_{\calb}\to\calb$ with fibers $\PP^{4m-1}$
such that $\mathbf P_{\boldsymbol{\widetilde U}}\simeq \boldsymbol{\widetilde U}
\times_{\calb}\mathbf P_{\calb}$. Indeed, as follows from \eqref{twisted-cocycle},
the failure of the transition functions of $\mathbf E_{\boldsymbol{\widetilde U}}$
and of $\boldsymbol\tau$ to form a 1-cocycle for the étale open
covering $\boldsymbol{\widetilde U}\to\calb$ is only in a scalar factor $c_g$,
so the transition functions modulo homotheties do form a 1-cocycle and define
a projective bundle over $\calb$.

The open subset $Z_{\tilde u}$ in 
$\PP^{4m-1}_{\tilde u}:=\PP\Hom(\widetilde{\mathbf E}_{\Gamma_{\tilde u}},\mathbf L(2)|_{\Gamma_{\tilde u}})$
consisting of proportionality classes 
$[\phi]$ of surjective homomorphisms $\phi:\widetilde{\mathbf E}_{\Gamma_{\tilde u}}
\onto \mathbf L(2)|_{\Gamma_{\tilde u}}$ is the complement to the
zero locus $\Delta_{\tilde u}$ of the resultant $\calr(f_0,f_1)$ of the two components of
$\phi$, viewed as a pair of polynomials $(f_0,f_1)$ on
$\Gamma_{\tilde u}\simeq \PP^1$ of degree $2m-1$, $f_i\in
H^0(\Gamma_{\tilde u}, \calo_{\Gamma_{\tilde u}}((2m-1)\pt))$.
It is easy to see that $\calr(f_0,f_1)$ depends on the choice
of homogeneous coordinates on $\Gamma_{\tilde u}\simeq \PP^1$ via the character
$GL_2(\CC)\to\CC^*,\ A\mapsto (\det A)^{(2m-1)^2}$, and is multiplied by $(\det B)^{-4m+2}$ when the identification 
$\widetilde{\mathbf E}_{\Gamma_{\tilde u}}
\simeq \calo^{\oplus 2}_{\Gamma_{\tilde u}}$
is changed by means of a matrix $B\in GL_2(\CC)$.
This implies that the resultant $\calr$ can be viewed as a section of a line bundle
upon lifting to an appropriate $GL_2(\CC)\times GL_2(\CC)$-torsor over
$\mathbf P_{\boldsymbol{\widetilde U}}$, and the zero locus of this section
descends to a relative divisor $\boldsymbol{\Delta}_\calb$ in
$\mathbf P_{\calb}$ over $\calb$.

We have thus shown that the surjectivity loci
$Z_{\tilde u}\subset \PP^{4m-1}_{\tilde u}$ glue
into an open subset $\calz_{\boldsymbol{\widetilde U}}
\subset \mathbf P_{\boldsymbol{\widetilde U}}$ which is the
inverse image of the open subset $\calz_\calb=\mathbf P_{\calb}\setminus
\boldsymbol{\Delta}_\calb$ of~$\mathbf P_{\calb}$. Now we are ready to construct a universal family $\mathbf F$ of the instanton sheaves of charge $n$
over $\calz_{\boldsymbol{\widetilde U}}$.

Remark, that $\boldsymbol{p_{{\widetilde U}*}}(\calo_{\mathbf P_{\boldsymbol{\widetilde U}}/
\boldsymbol{\widetilde U}}(1))=\tau^\dual$ (the equality sign meaning a canonical isomorphism),
so that the canonical Casimir section $\boldsymbol\sigma\in\Gamma(\boldsymbol{\widetilde U},
\boldsymbol\tau^\dual\otimes\boldsymbol\tau)$ comes from a unique section
$\boldsymbol{\tilde\sigma}\in \Gamma(\mathbf P_{\boldsymbol{\widetilde U}},
\boldsymbol{p_{\widetilde U}^*}\tau\otimes
\calo_{\mathbf P_{\boldsymbol{\widetilde U}}/
\boldsymbol{\widetilde U}}(1))$. The latter defines a morphism of sheaves
$\boldsymbol{\phi}:\boldsymbol{p_{\widetilde U}^*}\mathbf E_{\boldsymbol{\widetilde U}}
\to\boldsymbol{p_{\widetilde U}^*}\boldsymbol\iota_*\mathbf L(2)$, and we set
$\mathbf F:=\ker \boldsymbol{\phi}|_{\calz_{\boldsymbol{\widetilde U}}}$. We have the following exact triple:
$$
0\lra \mathbf F \lra \boldsymbol{p_{\widetilde U}^*}\mathbf E_{\boldsymbol{\widetilde U}}
|_{\calz_{\boldsymbol{\widetilde U}}}
\xrightarrow{\boldsymbol{\phi}|_{\calz_{\boldsymbol{\widetilde U}}}} 
\boldsymbol{p_{\widetilde U}^*}(\boldsymbol\iota_*\mathbf L(2))\otimes
\calo_{\mathbf P_{\boldsymbol{\widetilde U}}/
\boldsymbol{\widetilde U}}(1)
|_{\calz_{\boldsymbol{\widetilde U}}} \lra 0\ .
$$
By construction, the classifying map ${\calz_{\boldsymbol{\widetilde U}}}\to\calm(n)$ of $\mathbf F$ factors through a bijection $\calb\to\cald(m,n)$, and we are done.
\end{proof}

Now we will prove the smoothness of $\calm(n)$ along $\mathcal{D}(m,n)$.

\begin{proposition}\label{prop 4.1}
Let $[F]\in\cald(m,n)$ with $n\ge1$ and $1\le m\le n$, where $[F]$ denotes,
as before, the point of the moduli space representing the isomorphism class of a stable sheaf. Then $\dim {\rm Ext}^1(F,F)=8n-3$ and
${\rm Ext}^2(F,F)=0$, i.e. $[F]$ is a nonsingular point of $\calm(n)$.
\end{proposition}

\begin{proof}

Recall that, if $\iota:\Gamma\hookrightarrow\p3$ is a smooth rational curve of degree $m$ belonging to $R^*_0(m)$, then by definition
\begin{equation}\label{nbdl m}
N_{\Gamma/\p3}\simeq
\left\{\begin{array}{lll}
\mathcal{O}_\Gamma((2m-1)\pt )^{\oplus 2} & \mbox{if} & m\ne2\\
\mathcal{O}_\Gamma(2\pt )\oplus\mathcal{O}_\Gamma(4\pt )& \mbox{if} & m=2.
\end{array}\right.
\end{equation}
Thus
\begin{equation}\label{det nbdl}
\det N_{\Gamma/\p3}\simeq \mathcal{O}_\Gamma((4m-2)\pt )\ \mbox{for all} \ m\ge1.
\end{equation}

Moreover,
\begin{equation}\label{inexts1}
\inext^1(\iota_*\mathcal{O}_\Gamma,\iota_*\mathcal{O}_\Gamma) \simeq \iota_*N_{\Gamma/\p3} ~~,~~
\inext^2(\iota_*\mathcal{O}_\Gamma,\iota_*\mathcal{O}_\Gamma) \simeq \iota_*\det N_{\Gamma/\p3} ,
\end{equation}
\begin{equation}\label{inexts2}
\inext^1(\iota_*\mathcal{O}_\Gamma,\op3) = 0 ~~{\rm and}~~
\inext^2(\iota_*\mathcal{O}_\Gamma,\op3) \simeq \iota_*\det N_{\Gamma/\p3} .
\end{equation}

We have the short exact sequence
\begin{equation}\label{ratcurve-ses}
0 \to F \to F^{\dual\dual} \to (\iota_*L)(2) \to 0 ,
\end{equation}
where $L:=\mathcal{O}_\Gamma(-\pt )$ and $F^{\dual\dual}$ is a rank $2$ locally free instanton sheaf of charge $n-m$. It follows that $\ext^2(F^{\dual\dual},F^{\dual\dual})=0$ (see \cite{JV}) and
$$ H^1(\p3,(\iota_*L^2)(4)) \simeq H^1(\Gamma, \mathcal{O}_\Gamma((4m-2)\pt )) = 0, $$
thus we can apply Lemma \ref{global exts}.

We start by computing $h^1(\inhom(F,F))$. Since
$$ h^0(\p3,(\iota_*L^2)(4)) = h^0(\Gamma,\mathcal{O}_\Gamma((4m-2)\pt )) = 4m-1, $$
we obtain from formula (\ref{inhom dim}) that either
\begin{equation}\label{dims1}
h^1(\inhom(F,F)) = 8n-4m-4, ~~{\rm if}~~ n>m,
\end{equation}
i.e. if $F^{\dual\dual}$ is nontrivial, or
\begin{equation}\label{dims2}
h^1(\inhom(F,F)) = 4m-4, ~~{\rm if}~~ n=m,
\end{equation}
i.e. if $F^{\dual\dual}$ is trivial.

Next, we characterize the sheaf $\inext^1(F,F)$. To do this, we first apply the contravariant functor
$\inhom(-,F^{\dual\dual})$ to the sequence (\ref{ratcurve-ses}). Since $\inhom(\iota_*L,F^{\dual\dual})=0$
(the first sheaf is torsion, while the second is locally free), one concludes that
$\inhom(F^{\dual\dual},F^{\dual\dual})\simto\inhom(F,F^{\dual\dual})$
and
$$ \inext^1((\iota_*L)(2),F^{\dual\dual})\simeq\inext^1(\iota_*\mathcal{O}_\Gamma,\op3)\otimes
(F^{\dual\dual}\otimes(\iota_*L^{-1}))(-2)=0, $$
by (\ref{inexts2}).

Notice that $\inhom(F,(\iota_*L)(2))\simeq\iota_*\inhom(\iota^*F,\mathcal{O}_\Gamma((2m-1)\pt ))$.
Applying the functor $\iota^*(-\otimes\iota_*\mathcal{O}_\Gamma)$ to the sequence (\ref{ratcurve-ses}) one obtains the following exact sequence of sheaves on $\Gamma$:
\begin{equation}\label{sqc2}
0 \to \iota^*\intor_1((\iota_*L)(2),\iota_*\mathcal{O}_\Gamma) \to F|_{\Gamma} \to F^{\dual\dual}|_{\Gamma} \to
\mathcal{O}_\Gamma((2m-1)\pt ) \to 0.
\end{equation}
Breaking it into two short exact sequences, one obtains, since
$F^{\dual\dual}|_{\Gamma}\simeq\mathcal{O}_\Gamma^{\oplus 2}$:
\begin{equation}\label{sqc3}
0\to\iota^*\intor_1((\iota_*L)(2),\iota_*\mathcal{O}_\Gamma) \to F|_{\Gamma} \to G \to 0
\end{equation}
and
\begin{equation}\label{sqc4}
0 \to G \to \mathcal{O}_\Gamma^{\oplus 2} \to \mathcal{O}_\Gamma((2m-1)\pt ) \to 0 .
\end{equation}
One concludes from sequence (\ref{sqc4}) that $G\simeq\mathcal{O}_\Gamma(-(2m-1)\pt )$. Moreover, since by (\ref{nbdl m})
$$ \iota^*\intor_1((\iota_*L)(2),\iota_*\mathcal{O}_\Gamma)\simeq N_{\Gamma/\p3}^\dual\otimes\mathcal{O}_\Gamma((2m-1)\pt )
\simeq\mathcal{O}_\Gamma^{\oplus 2}\ \mbox{for} \ m\ne2, $$
and
$$ \iota^*\intor_1((\iota_*L)(2),\iota_*\mathcal{O}_\Gamma)\simeq N_{\Gamma/\p3}^\dual\otimes\mathcal{O}_\Gamma(3\pt )
\simeq\mathcal{O}_\Gamma(\pt)\oplus\mathcal{O}_\Gamma(-\pt )\ \mbox{for} \ m=2, $$
it follows from sequence (\ref{sqc3}) that
$$
F|_{\Gamma} \simeq
\mathcal{O}_\Gamma^{\oplus 2} \oplus \mathcal{O}_\Gamma(-(2m-1)\pt )\ \mbox{for} \ m\ne2,
$$
$$
F|_{\Gamma} \simeq \mathcal{O}_\Gamma(\pt ) \oplus \mathcal{O}_\Gamma(-\pt ) \oplus \mathcal{O}_\Gamma(-3\pt )\ \mbox{for} \ m=2.
$$
We then conclude that
\begin{equation}\label{id1}
\inhom(F,(\iota_*L)(2))\simeq\iota_*\mathcal{O}_\Gamma((2m-1)\pt )^{\oplus 2}\oplus
\iota_*\mathcal{O}_\Gamma((4m-2)\pt )\ \mbox{for} \ m\ne2,
\end{equation}
\begin{equation}\label{id1m=2}
\inhom(F,(\iota_*L)(2))\simeq\iota_*\mathcal{O}_\Gamma(2\pt )\oplus
\iota_*\mathcal{O}_\Gamma(4\pt ) \oplus\iota_*\mathcal{O}_\Gamma(6\pt )\ \mbox{for} \ m=2.
\end{equation}

Now, by applying the functor $\inhom(-,F^{\dual\dual})$ to (\ref{ratcurve-ses}), one
concludes that $\inext^1(F,F^{\dual\dual})\simeq\inext^2((\iota_*L)(2),F^{\dual\dual})$. However,
\begin{multline*}
\inext^2((\iota_*L)(2),F^{\dual\dual})  \simeq
\inext^2(\iota_*\mathcal{O}_\Gamma(-\pt ),\op3)\otimes F^{\dual\dual}(-2) \simeq \notag \\
 \inext^2(\iota_*\mathcal{O}_\Gamma,\op3) \otimes \iota_*\mathcal{O}_\Gamma(\pt ) \otimes F^{\dual\dual}(-2) . 
\end{multline*}
Thus, by (\ref{det nbdl}), (\ref{inexts2}) and since $F^{\dual\dual}|_{C}\simeq\mathcal{O}_\Gamma^{\oplus 2}$, we obtain
\begin{equation}\label{id2}
\inext^1(F,F^{\dual\dual}) \simeq \iota_*\mathcal{O}_\Gamma((2m-1)\pt )^{\oplus 2}.
\end{equation}

Finally, we apply the functor $\inhom(-,(\iota_*L)(2))$ to (\ref{ratcurve-ses}), using (\ref{inexts1})
and the isomorphism
$\inext^2((\iota_*L)(2),(\iota_*L)(2)) \simeq \inext^2(\iota_*\mathcal{O}_\Gamma,\iota_*\mathcal{O}_\Gamma)$,
and we see that
\begin{equation}\label{id30}
\inext^1(F,(\iota_*L)(2))\simeq\inext^2((\iota_*L)(2),(\iota_*L)(2))\simeq\iota_*\det N_{\Gamma/\p3}.
\end{equation}
Thus, in view of (\ref{det nbdl}), we have
\begin{equation}\label{id3}
\inext^1(F,(\iota_*L)(2)) \simeq \iota_*\mathcal{O}_\Gamma((4m-2)\pt ).
\end{equation}

Substituting (\ref{id1})-(\ref{id2}) and (\ref{id3}) into (\ref{inhom sqc 2}), we obtain
the exact sequence
\begin{multline*}
0 \to \iota_*\mathcal{O}_\Gamma((4m-2)\pt )  \to
\iota_*\mathcal{O}_\Gamma((2m-1)\pt )^{\oplus 2}\oplus\iota_*\mathcal{O}_\Gamma((4m-2)\pt ) \to
\\ \inext^1(F,F) \to  
 \iota_*\mathcal{O}_\Gamma((2m-1)\pt )^{\oplus 2} \to \iota_*\mathcal{O}_\Gamma((4m-2)\pt ) \to 0\ \mbox{for} \ m\ne2, 
\end{multline*}
and
\begin{multline*}
0 \to \iota_*\mathcal{O}_\Gamma(6\pt )  \to
\iota_*\mathcal{O}_\Gamma(2\pt )\oplus\iota_*\mathcal{O}_\Gamma(4\pt )\oplus
\iota_*\mathcal{O}_\Gamma(6\pt ) \to
\inext^1(F,F) \to \\ 
 \iota_*\mathcal{O}_\Gamma(3\pt )^{\oplus 2} \to \iota_*\mathcal{O}_\Gamma(6\pt ) \to 0\ \mbox{for} \ m=2. 
\end{multline*}

Breaking these into three short exact sequences, one obtains:
\begin{equation}\label{i}
0\to G_1 \to \iota_*\mathcal{O}_\Gamma((2m-1)\pt )^{\oplus 2} \to
\iota_*\mathcal{O}_\Gamma((4m-2)\pt ) \to 0\ \mbox{for} \ m\ge1,
\end{equation}
\begin{equation}\label{ii}
0\to G_2 \to \inext^1(F,F) \to G_1 \to 0,
\end{equation}
\begin{multline}\label{iii}
0\to \iota_*\mathcal{O}_\Gamma((4m-2)\pt ) \to \\ 
\iota_*\mathcal{O}_\Gamma((2m-1)\pt )^{\oplus 2}\oplus\iota_*\mathcal{O}_\Gamma((4m-2)\pt )
\to G_2 \to 0\ \mbox{for}\ m\ne2,
\end{multline}
and
\begin{equation}\label{iiim=2}
0 \to \iota_*\mathcal{O}_\Gamma(6\pt ) \to \iota_*\mathcal{O}_\Gamma(2\pt )\oplus
\iota_*\mathcal{O}_\Gamma(4\pt )\oplus\iota_*\mathcal{O}_\Gamma(6\pt ) \to G_2 \to 0\ \mbox{for} \ m=2.
\end{equation}
From (\ref{i}) one concludes that $G_1\simeq \iota_*\mathcal{O}_\Gamma$, while (\ref{iii}) and
(\ref{iiim=2}) impliy that $G_2\simeq \iota_*\mathcal{O}_\Gamma((2m-1)\pt )^{\oplus 2}$ for $m\ne2$ and $G_2\simeq \iota_*\mathcal{O}_\Gamma(2\pt )\oplus\iota_*\mathcal{O}_\Gamma(4\pt )$ for $m=2$.
Then (\ref{ii}) yields
$$
\inext^1(F,F) \simeq \iota_*\mathcal{O}_\Gamma((2m-1)\pt )^{\oplus 2}\oplus\iota_*\mathcal{O}_\Gamma,
\ \ \ \ \ m\ne2,
$$
$$
\inext^1(F,F) \simeq \iota_*\mathcal{O}_\Gamma(2\pt )\oplus\iota_*\mathcal{O}_\Gamma(4\pt )
\oplus\iota_*\mathcal{O}_\Gamma,\ \ \ \ \ m=2.
$$

It then follows that
$$
h^0(\inext^1(F,F)) = 4m+1 ~~~~ {\rm and} ~~~~ h^1(\inext^1(F,F)) = 0.
$$
We conclude from Lemma \ref{global exts} that $\ext^2(F,F)=0$ and
$$
\dim\ext^1(F,F) = 4m+1 + 8n-4m-4 = 8n-3 ~~{\rm if}~~ n>m ,
$$
$$
\dim\ext^1(F,F) = 4m+1 + 4m-4 = 8n-3 ~~{\rm if}~~ n=m ,
$$
as desired.
\end{proof}

\bigskip

\section{Sheaves in $\cald(m,n)$ as limits of instantons}\label{lim-inst}

The goal of this section is to prove that the varieties $\cald(m,n)$ are contained in the instanton boundary $\partial\cali(n)$. We begin with the following technical lemma.

\begin{lemma}\label{F,G}
Let $C $ be a smooth irreducible curve with a marked point $0$, and set $\mathbf{B}:=C \times\p3$. Let $\mathbf{F}$ and $\mathbf{G}$ be $\mathcal{O}_{\mathbf{B}}$-sheaves, flat over $C $ and such that $\mathbf{F}$ is locally free along $\mathrm{Supp}(\mathbf{G})$. Denote $\p3_t=\{t\}\times\p3$, and
$$
G_t=\mathbf{G}|_{\{t\}\times\p3},\ \ \ F_t=\mathbf{F}|_{\{t\}\times\p3}\ \ {\rm for}\ t\in C . $$ Assume that, for each $t\in C $,
\begin{equation}\label{vanish Hi}
H^i(\inhom(F_t,G_t))=0,\ \ \ i\ge1.
\end{equation}
Assume that $s:F_0\to G_0$ is an epimorphism. Then, after possibly shrinking $C $,
$s$~extends to an epimorhism $\mathbf{s}:\mathbf{F}\twoheadrightarrow\mathbf{G}$.
\end{lemma}

\begin{proof}
The condition that $\mathbf{F}$ is locally free along $\mathrm{Supp}(\mathbf{G})$ implies that
$\inext^1(\mathbf{F},\mathbf{G})=0$. Thus applying the functor $\inhom(\mathbf{F},-)$
to the short exact sequence
$$ 0\to\mathbf{G}\overset{\cdot\zeta_t}\to\mathbf{G}\to G_t\to0,\ t\in C , $$
where $\zeta_t$ is the local parameter of $\calo_{C ,t}$,
we obtain the short exact sequence
\begin{equation}\label{trA}
0 \to \mathbf{H}\overset{\cdot\zeta_t} \to \mathbf{H} \to \inhom(\mathbf{F},G_t) \to 0,\ \ \ t\in C ,
\end{equation}
in which $\mathbf{H}$ denotes the sheaf $\inhom(\mathbf{F},\mathbf{G})$. 
Since $\mathbf{F}$ is locally free along $\mathrm{Supp}(G_t)\subset\mathrm{Supp}(\mathbf{G})$, we obtain that
\begin{equation}\label{2 isoms}
\inhom(\mathbf{F},G_t)\simeq\inhom(F_t,G_t)\simeq\mathbf{H}|_{\p3_t},\ \ \ t\in C .
\end{equation}
This means that (\ref{trA}) can be rewritten as the short exact sequence
\begin{equation}\label{trB}
0\to\mathbf{H}\overset{\cdot\zeta_t}\to\mathbf{H}\to\mathbf{H}|_{\p3_t}\to0
\end{equation}
obtained by applying the functor $-\otimes\mathbf{H}$ to the short exact sequence
$$ 0 \to \mathcal{O}_{\mathbf{B}} \overset{\cdot\zeta_t} \to \mathcal{O}_{\mathbf{B}} \to
\mathcal{O}_{\mathbb{P}^3_t} \to 0. $$
This yields $\intor_1(\mathbf{H},\mathcal{O}_{\p3_t})=0$ for any $t\in C $.

The last equality means that the $\mathbf{H}$ is flat over $C $. Furthermore, the condition (\ref{vanish Hi}) in view of (\ref{2 isoms}) can be rewritten as $H^i(\inhom(F_t,G_t))=0,\ \ \ i\ge1.$
Thus, by Base Change, we obtain that $R^jp_*\mathbf{H}=0$ for $j\ge1$, where $p:\mathbf{B}\to C $ is the projection. As a corollary, applying $R^jp_*$ to (\ref{trB}) for $t=0$ and using (\ref{2 isoms}), we obtain 
$$ 0\to p_*\mathbf{H}\to p_*\mathbf{H}\overset{e}\to p_*\inhom(F_0,G_0) \to 0. $$
Since $ p_*\inhom(F_0,G_0)$ is supported at $0$, it follows that, after possibly shrinking $C $, the homomorphism of sections
$$ h^0(e):H^0(p_*\mathbf{H})\to H^0(p_*\inhom(F_0,G_0)) $$
is surjective. However, this epimorphism canonically coincides with the epimorphism
$\Hom(\mathbf{F},\mathbf{G}) \twoheadrightarrow \Hom(F_0,G_0)$. Hence, the homomorphism
$s\in{Hom}(F_0,G_0)$, from the hypothesis of the lemma extends to a morphism
$\mathbf{s}\in\mathrm{Hom}(\mathbf{F},\mathbf{G})$. In addition, as $s$ is an epimorphism, again, shrinking $C $ if necessary, we may assume that $\mathbf{s}$ is an epimorphism as well.
\end{proof}

\begin{proposition}\label{Hooft}
$\overline{\cald(1,n)}\subset\partial\cali(n)$ for $n\ge1$.
\end{proposition}

\begin{proof}
Fix a disjoint union of $n+1$ lines $\Lambda=\underset{0\le i\le n}\bigsqcup \ell_i$ in $\p3$. The extensions of $\op3$-sheaves of the form
\begin{equation}\label{trC}
0 \to \op3(-1) \to E \to I_{\Lambda,\p3}(1) \to 0
\end{equation}
are classified by the vector space $V=\mathrm{Ext}^1(I_{\Lambda,\p3}(1),\op3(-1))$. 
It is easy to see that
any nontrivial
extension \eqref{trC}
defines an instanton
sheaf $E$; when $E$ is
locally free, it is
called a {\it 't Hooft instanton}.
A standard computation gives an isomorphism
\begin{equation}\label{1 isom}
V\simeq\underset{0\le i\le n}\oplus H^0(\mathcal{O}_{\ell_i}).
\end{equation}
Under this isomorphism, any point $x$ in $V$ can be represented by its coordinates
$x=(t_0,...,t_n),\ t_i\in H^0(\mathcal{O}_{\ell_i})=\C$. 
By \cite[Proposition 3.1]{HH1}, there exists a universal (flat) family of extensions (\ref{trC}) parametrized by 
the affine space $V=\mathbb{A}^{n+1}$. Restrict it to the line defined by $\mathbb{A}^1=\{(t,1,...,1)\ |\ t\in\C\}$, and denote by
$\mathbf{E}=\{E_t\}_{t\in\mathbb{A}^1}$ the thus obtained family of instanton sheaves with base $\mathbb{A}^1$. It follows from the construction of the isomorphism (\ref{1 isom}) that $E_t$ is locally free and $[E_t]\in\cali(n)$ for $0\ne t\in\mathbb{A}^1$. For $t=0$, the sheaf $E_0$ is not locally free. By Theorem \ref{jg-thm}, $E_0^{\dual\dual}$ is an instanton bundle, $[E_0^{\dual\dual}]\in\cali(n-1)$. In fact, it fits into the short exact sequence
\begin{equation}\label{trD}
0 \to \op3(-1) \to E_0^{\dual\dual} \to I_{\Lambda_1,\p3}(1) \to 0, \ \ \
\Lambda_1=\underset{1\le i\le n}\bigsqcup \ell_i,
\end{equation}
so it is a 't Hooft instanton.

We thus obtain a morphism $\phi:\mathbb{A}^1 \to \overline{\cali(n)}, \ t\mapsto[E_t]$, and this proves that
$$
[E_0]\in\partial\cali(n).
$$
The exact triples (\ref{trC}) for $E=E_0$ and (\ref{trD}) yield the exact triple
$$ 0 \to E_0\to E_0^{\dual\dual} \to \mathcal{O}_{\ell_0}(\pt ) \to 0, $$
which shows that $[E_0]\in\mathcal{D}(1,n)$. By Proposition \ref{prop 4.1}, $\mathrm{Ext}^2(E_0,E_0)=0$. So $[E_0]$ is a smooth point of $\calm(n)$, in particular,
$[E_0]$ does not lie in the intersection of two components of $\calm(n)$. Hence
$\cald(1,n)$, $\overline{\mathcal{D}(1,n)}$ are entirely contained in the component
$\overline{\mathcal{I}(n)}$. As none of the sheaves in $\cald(1,n)$
is locally free, $\cald(1,n)\subset\overline{\mathcal{D}(1,n)}\subset
\partial\cali(n)$.
\end{proof}

Let now $n\geq m\ge2$. Fix a disjoint union $\Gamma_1=\Gamma'\sqcup\ell_1$ in $\p3$, where $\Gamma'$ is a smooth irreducible rational curve of degree $m-1$ and $\ell_1$ is a line, and let $\iota_1:\Gamma_1\hookrightarrow\p3$ be the embedding. Consider the $\mathcal{O}_{\Gamma_1}$-sheaf
$L_1=\mathcal{O}_{\Gamma'}(-\pt )\oplus\mathcal{O}_{\ell_1}(-\pt )$.
Fix an instanton bundle $E\in\cali(n-m)$ such that
\begin{equation}\label{E restr on}
E|_{\Gamma'}\simeq\mathcal{O}_{\Gamma'}^{\oplus2},\ \ \ \ E|_{\ell_1}\simeq\mathcal{O}_{\ell_1}^{\oplus2};
\end{equation}
the existence of such a bundle $E$ follows from the above. Let us fix an epimorphism
$e_1:E\twoheadrightarrow (\iota_{1*}L_1)(2)$, and let the sheaf $F_1$ be defined by the short exact sequence
\begin{equation}\label{def E1}
0 \to F_1 \to E\xrightarrow{e_1}  (\iota_{1*}L_1)(2) \to 0.
\end{equation}

\begin{proposition}\label{C-def}
Let $n\geq m\ge2$, and assume that $\overline{\cald(m-1,n-1)}\subset\partial\cali(n-1)$. Let $F_1$ be the
$\op3$-sheaf defined by (\ref{def E1}). Then $[F_1]\in\partial\cali(n)$.
\end{proposition}

\begin{proof}
Let $F'$ be the kernel of the composition
$$ \varepsilon_1: E \ontoo{e_1} (\iota_{1*}L)(2)\ontoo{e'}
\iota_{1*}\mathcal{O}_{\Gamma'}((2m-3)\pt ), $$
where $e'$ is the projection onto the direct summand. We thus obtain a short exact sequence
\begin{equation}\label{def E'}
0 \to F' \to E \xrightarrow{\varepsilon_1} \iota_{1*}\mathcal{O}_{\Gamma'}((2m-3)\pt ) \to 0,
\end{equation}
which shows, since $[E]\in\cali(n-m)$, that
$$
[F']\in\mathcal{D}(m-1,n-1).
$$
Furthermore, since $\Gamma_1$ is a disjoint union of $\Gamma'$ and $\ell_1$, the last isomorphism in (\ref{E restr on}) yields
\begin{equation}\label{E' restr on}
F'|_{\ell_1}\simeq\mathcal{O}_{\ell_1}^{\oplus2}.
\end{equation}
Besides, (\ref{def E1}) and (\ref{def E'}) imply the exact triple
\begin{equation}\label{E' triple}
0\to F_1 \to F' \overset{s'} \to \iota_{1*}\mathcal{O}_{\ell}(\pt )\to0.
\end{equation}

By hypothesis, $[F']$ is in $\partial{\cali(n-1)}$.
Hence we can find a smooth irreducible curve $ C$ with marked point 0 and a $\mathcal{O}_{ C\times\mathbb{P}^3}$-sheaf $\boldsymbol{F}'$, flat over $C$, such that
\begin{equation}\label{prop of E'}
[\boldsymbol{F}'|_{\{y\}\times\p3}]\in\cali(n-1),\ \ \ y\in  C\smallsetminus\{0\},\ \ \ \ \ \
\boldsymbol{F}'|_{\{0\}\times\p3}\simeq F'.
\end{equation}
By (\ref{E restr on}),  $F'$ is locally free along $\ell_1$, and by
(\ref{prop of E'}),  $\mathbf{E}:=\boldsymbol{F}'$ is locally free
along $\mathrm{Supp}(\mathbf{G})$, where $\mathbf{G}=\boldsymbol{\iota}_{1*}(\mathcal{O}_ C\boxtimes\mathcal{O}_{\ell_1}(\pt ))$
and $\boldsymbol{\iota}_1=\mathrm{id}\times(\iota_1|_{\ell_1}):  C\times\ell_1\hookrightarrow  C\times\p3$ is the embedding. 
Moreover, (\ref{E restr on}) implies that, after possibly shrinking $ C$, we may assume that $\mathbf{E}|_{\{y\}\times\ell_1}\simeq\mathcal{O}_{\ell_1}^{\oplus2}$ for all $y\in C$. This together with the definition of $\mathbf{G}$ yields (\ref{vanish Hi}). Hence $\mathbf{E}$, $\mathbf{G}$ satisfy the hypothesis of Lemma \ref{F,G}, and there exists an epimorphism
$\mathbf{s}'|:\mathbf{E}\twoheadrightarrow\mathbf{G}$ extending the epimorphism $s'$ in (\ref{E' triple}). We thus obtain the short exact sequence
\begin{equation}\label{bold E'}
0 \to \boldsymbol{F}_1 \to \boldsymbol{F}' \overset{\mathbf{s}'} \to
\boldsymbol{\iota}_{1*}(\mathcal{O}_ C\boxtimes\mathcal{O}_{\ell_1}(\pt ))\to0
\end{equation}
extending (\ref{E' triple}), that is
\begin{equation}\label{restr=E1}
\boldsymbol{F}_1|_{\{0\}\times\p3}\simeq F_1.
\end{equation}
Restricting (\ref{bold E'}) to $\{y\}\times\p3$ for $y\in C$,
we obtain, by (\ref{prop of E'}), that
$[\boldsymbol{F}_1|_{\{y\}\times\p3}]\in\cald(1,n)$ for $y\in  C\smallsetminus\{0\}$. Moreover, $F_1$ is stable by
Corollary \ref{rk2-mu-stable} and Lemma \ref{stable2}, so $[F_1]\in\calm(n)$.
This together with (\ref{restr=E1}) implies that $[F_1]\in\overline{\cald(1,n)}$, and the proposition
follows from Proposition \ref{Hooft}.
\end{proof}

Let now $n\geq m\geq 2$, $[E]\in \cali(n-m)$, let $\Gamma_0=\Gamma'\cup\ell_0$, where $\Gamma'$ is a smooth rational curve of degree $m-1$, $\ell_0$ is a line intersecting $\Gamma'$ quasi-transversely at one point, say, $x$, and let the following properties hold:
\begin{equation}\label{Gamma0}
N_{\Gamma'/\p3} \simeq \mathcal{O}_{\Gamma'}((2m-3)\pt )^{\oplus2},\ 
\end{equation}
\begin{equation}\label{E restr 2}
E|_{\Gamma'}\simeq\mathcal{O}_{\Gamma'}^{\oplus2},\  E|_{\ell_0}\simeq\mathcal{O}_{\ell_0}^{\oplus2}.
\end{equation}
Note that (\ref{Gamma0}) and the first two isomorphism in  (\ref{E restr 2}) mean, in the notation of Section \ref{ratcurves}, that $\Gamma'\in\calr^*_{0}(m-1)_E$.

We will now treat the transform $E_0$ of the instanton bundle $E$ along the non-invertible sheaf
$\mathcal{O}_{\ell} (-{\rm \pt})\oplus\mathcal{O}_{\Gamma'} (-{\rm \pt})$ on $\Gamma_0$.

\begin{proposition}\label{E0 in}
Let $n\geq m\geq 2$, and let $E$, $\Gamma_0$ be as above. Set
$L_0=\mathcal{O}_{\Gamma'}(-{\rm \pt} )\oplus\mathcal{O}_{\ell_0}(-{\rm \pt})$
and consider a sheaf $E_0$ fitting into an exact sequence
\begin{equation}\label{E0 triple0}
0 \to E_0 \to E \to (\iota_{0*}L_0)(2) \to 0,
\end{equation}
where $\iota_0:\Gamma_0\into\p3$ is the natural embedding.
Assume that $\overline{\cald(m-1,n-1)}\subset\partial\cali(n-1)$. Then $[E_0]\in\partial\cali(n)$.
\end{proposition}

\begin{proof}
We can include $\Gamma_0$ in a one-parameter family of curves
$\boldsymbol{\Gamma}_U=\{\Gamma_u=\Gamma'\cup\ell_u\}_{u\in U}$ over a smooth
irreducuble curve $U$ such that $\Gamma'\cap\ell_u=\varnothing$ for
any $u\in U\setminus\{0\}$. Shrinking $U$, if necessary, we may assume by semicontinuity and the last isomorphism in 
(\ref{Gamma0}) that
\begin{equation}\label{E|lu}
E|_{\ell_u}\simeq\mathcal{O}_{\ell_u}^{\oplus2},\ \ \ u\in U.
\end{equation}
The union $\boldsymbol{\ell}_U=\bigcup_{u\in U}\ell_u$ is a surface meeting $\Gamma'$ only at $x$. We 
may assume that this intersection is transverse and that the lines $\ell_u$ do not intersect each other (for example, we can start the construction from a family of lines $\ell_u$ which is one of the two families of rectilinear generators of a smooth quadric, and then $\boldsymbol{\ell}_U$ is an open subset of this quadric).
We have $\boldsymbol{\Gamma}_U=(U\times \Gamma')\cup \boldsymbol{\ell}_U$.
 
Consider the following $\mathcal{O}_{\boldsymbol{\Gamma}_U}$-sheaf:
\begin{equation}\label{L1}
\boldsymbol{L}_U = \mathcal{O}_U\boxtimes\mathcal{O}_{\Gamma'}(-\pt )\oplus
(\calo_{\p3}(-1)|_{\boldsymbol{\ell}_U}).
\end{equation}
Then for each $u\in U$, we have
$$ {L}_u:=\boldsymbol{L}_U|_{u\times\p3}\simeq\mathcal{O}_{\Gamma'}(-\pt )\oplus
\mathcal{O}_{\ell_u}(-\pt ). $$
Let $\boldsymbol{\jmath}:\boldsymbol{\Gamma}_U\hookrightarrow U\times\p3$ be the natural embedding.
Then $\boldsymbol{\jmath}|_{\{0\}\times\p3}= \iota_0$, and by (\ref{E|lu}), a surjection
$E\overset{e_0}\twoheadrightarrow \iota_{0*}L_{0}(2)$ from \eqref{E0 triple0} extends
to a surjection
$\mathbf{e}:\mathcal{O}_U\boxtimes E\twoheadrightarrow\boldsymbol{\jmath}_*\boldsymbol{L}_U(2)$,
where twisting a sheaf on $\boldsymbol{\Gamma}_U$ or on $U\times\p3$ by an integer $k$ means tensoring it by $\mathcal{O}_U\boxtimes \calo_{\p3}(k)$.
The $\mathcal{O}_{U\times\mathbb{P}^3}$-sheaf $\boldsymbol{E}_U:=\ker\mathbf{e}$ fits in the exact triple
$$
0 \to \boldsymbol{E}_U\to \mathcal{O}_U\boxtimes E \xrightarrow{\mathbf{e}}
\boldsymbol{\jmath}_*\boldsymbol{L}_U(2) \to 0.
$$
Let $E_u:=\boldsymbol{E}_U|_{\{u\}\times\p3}$ for  $u\in U$.
If $u\in U\setminus\{0\}$, then $\Gamma_u$ is a disjoint union $\Gamma'\sqcup\ell_u$, so the result of Proposition \ref{C-def} holds for $F_1=E_u$. Hence $ [E_u]\in\partial\cali(n)$ for all $u\in U\setminus\{0\}$. For $u=0$, the stability of $E_0$ is assured by
Corollary \ref{rk2-mu-stable} and Lemma \ref{stable2}, so that $[E_0]\in\calm(n)$, and by closedness of $\partial\cali(n)$, we also have
$
[E_0]\in\partial\cali(n).
$
\end{proof}

\begin{remark}\label{nonflat} \rm 
The family of curves $\Gamma_u$ used in the proof is not flat, but it carries a $U$-flat family
of pure rank-1 sheaves $L_u$ that is a part of degeneration data of
a $U$-flat family of instanton sheaves $E_u$.
\end{remark}

Note that, by construction, $E_0$ fits in the exact triple
\begin{equation}\label{E0 triple}
0 \to E_0 \to E \to \iota_*(\mathcal{O}_{\Gamma'}((2m-3)\pt )\oplus\mathcal{O}_{\ell_0}(\pt )) \to 0.
\end{equation}

\begin{proposition}\label{smooth in E0}
Under the hypothesis of Proposition \ref{E0 in}, $\mathrm{Ext}^2(E_0,E_0)=0$, so that $[E_0]$ is a smooth point of $\mathcal{M}(n)$.
\end{proposition}

\begin{proof}
Denote $L':=\iota_*(\mathcal{O}_{\Gamma'}((2m-3)\pt))$ and $L_0:=\iota_*(\mathcal{O}_{\ell_0}(\pt))$. Then (\ref{E0 triple}) implies the exact triples
\begin{equation}\label{E0 triple A}
0\to F\to E\to L'\to0,
\end{equation}
\begin{equation}\label{E0 triple B}
0\to E_0\to F\to L_0\to0,
\end{equation}
where $F$ is the kernel of the composition
$E\twoheadrightarrow L'\oplus L_0
\ontoo{pr_1}L'$.
Applying to (\ref{E0 triple B}) the functor $\mathrm{Ext}^{\scriptscriptstyle\bullet}(-,E_0)$ 
we obtain the exact sequence
\begin{equation}\label{E0 triple C}
\mathrm{Ext}^2(F,E_0)\to\mathrm{Ext}^2(E_0,E_0)\to
\mathrm{Ext}^3(L_0,E_0),
\end{equation}
By Serre--Grothendieck duality, 
$\mathrm{Ext}^3(L_0,E_0)=
\mathrm{Hom}(E_0,L_0(-4{\pt}))^{\dual}$.
Applying $\inhom(-,L_0(-4{\pt}))$ to \eqref{E0 triple} and using (\ref{E restr 2}), we obtain the exact sequence
$$
0\to\calo_{\ell_0}(-4\pt)\to2\calo_{\ell_0}(-3\pt)\to
\inhom(E_0, \calo_{\ell_0}(-3\pt))\to
\calo_{\ell_0}(-3\pt)^{\oplus2}\oplus \mathbf{k}(x)\to 0
$$
where $\mathbf{k}(x)\simeq \inext^1(L',L_0(-4\pt))
\simeq \inext^2(L',L_0(-4\pt))$ is a skyscraper of length~1 at the point $x=\Gamma'\cap\ell_0$.
Since $\calo_{\ell_0}(-3\pt)$ has no torsion as an $\calo_{\ell_0}$-module,
$\inhom(E_0, \calo_{\ell_0}(-3\pt))$ is also a torsion free $\calo_{\ell_0}$-module, 
hence $\inhom(E_0, \calo_{\ell_0}(-3\pt))\simeq
\calo_{\ell_0}(-\pt)\oplus\calo_{\ell_0}(-3\pt)^{\oplus2}$ and
\begin{equation}\label{Ext3=0}
\mathrm{Ext}^3(L_0,E_0)=
\mathrm{Hom}(E_0,L_0(-4{\pt}))^{\dual}=H^0(\inhom(E_0, \calo_{\ell_0}(-3\pt)))^\dual=0.
\end{equation}
Next, since $L'$ has
homological dimension 2 and $E$ is locally free, $F$ has homological dimension 1. Therefore,
\begin{equation}\label{Ext2(F,E0)=0}
\inext^2(F,E_0)=0.
\end{equation}
Also, since $E$ is locally free, (\ref{E0 triple A}) implies that
\begin{equation}\label{Ext1(F,E0)=}
\inext^1(F,E_0)=\inext^2(L',E_0).
\end{equation}
Next, (\ref{Gamma0}) implies that
\begin{equation}\label{two inexts}
\begin{split}
&\inext^1(L',L')\simeq N_{\Gamma'/\mathbb{P}^3}\simeq
\mathcal{O}_{\Gamma'}((2m-3){\pt})^{\oplus2},\\
&\inext^2(L',L')\simeq 
\det N_{\Gamma'/\mathbb{P}^3}\simeq
\mathcal{O}_{\Gamma'}((4m-6){\pt}).
\end{split}
\end{equation}
Similarly, by
(\ref{E restr 2}), we have
\begin{equation}\label{Ext2(L',E)=}
\inext^2(L',E)\simeq
\det N_{\Gamma'/\mathbb{P}^3}
\otimes((L')^{\oplus2})^{\dual}\simeq
\mathcal{O}_{\Gamma'}((2m-3){\pt})^{\oplus2}.
\end{equation}
Now applying $\inext^{\scriptscriptstyle\bullet}(L',-)$ to (\ref{E0 triple}),
we obtain the exact sequence
$$
\inext^1(L',L'\oplus L_0)
\overset{\gamma}\to\inext^2(L',E_0)\to
\inext^2(L',E)\overset{\delta}\to
\inext^2(L',L'\oplus L_0).
$$
Using (\ref{two inexts}), (\ref{Ext2(L',E)=}) and the relations  
$\inext^1(L',L_0)\simeq\inext^2(L',L_0)\simeq\mathbf{k}(x)$,
we rewrite the last sequence as follows:
\begin{multline*}
\mathcal{O}_{\Gamma'}((2m-3){\pt})^{\oplus2}\oplus
\mathbf{k}(x)
\overset{\gamma}\to\inext^2(L',E_0)\to
\mathcal{O}_{\Gamma'}((2m-3){\pt})^{\oplus2}\overset{\delta}\to\\
\mathcal{O}_{\Gamma'}((4m-6){\pt})\oplus\mathbf{k}(x)\to 0.
\end{multline*}
It implies that $\ker(\delta)$ is an $\mathcal{O}_{\Gamma'}$-sheaf of
degree $\geq -1$, so $h^1(\ker(\delta))=0$. Similarly, 
$h^1(\mathrm{im}(\gamma))=0$. Hence 
$H^1(\inext^2(L',E_0))=0$, and by (\ref{Ext1(F,E0)=}),
\begin{equation}\label{h1(Ext1(F,E0))=0}
H^1(\inext^1(F,E_0))=0.
\end{equation}
Next, as $E$ is locally free, (\ref{E0 triple}) yields the exact triple
$$ H^1(\inhom(E,L'\oplus L_0))\to
H^2(\inhom(E,E_0))\to H^2(\inhom(E,E)). $$
From (\ref{E restr 2}) and the definition of $L'$ and $L_0$, it follows that
$H^1(\inhom(E,L'\oplus L_0))=0$. Besides,
since $[E]\in\mathcal{I}(n-m)$, it follows that $H^2(\inhom(E,E))=0$.
Thus we have
\begin{equation}\label{H2(hom(E,E_0))=0}
H^2(\inhom(E,E_0))=0.
\end{equation}
Applying $\inhom(-,E_0)$ to (\ref{E0 triple A}),
we obtain the exact triple
$0\to\inhom(E,E_0)\to\inhom(F,E_0)\overset{\varepsilon}{\to}
\inext^1(L',E_0)$. 
As $\mathrm{Supp}(L')=\Gamma'$, it follows that
$H^2(\mathrm{im}(\varepsilon))=0$, and the last exact triple together
with (\ref{H2(hom(E,E_0))=0}) yields
$H^2(\inhom(F,E_0))=0$.
This equality together with (\ref{Ext2(F,E0)=0}), (\ref{h1(Ext1(F,E0))=0}) and the spectral sequence
$E^{pq}_2=H^p(\inext^q(F,E_0))\Rightarrow
\mathrm{Ext}^{\scriptscriptstyle\bullet}(F,E_0)$ implies that
$\mathrm{Ext}^2(F,E_0)=0$.
By (\ref{E0 triple C}) and (\ref{Ext3=0}), we also have
$\mathrm{Ext}^2(E_0,E_0)=0$.
\end{proof}

\begin{proposition}\label{Dmn in In}
Let $n\geq m\ge2$, and assume that $\overline{\cald(m-1,n-1)}\subset\partial\cali(n-1)$. Then
$\overline{\cald(m,n)}\subset\cali(n)$.
\end{proposition}

\begin{proof}
Take $E$, $\Gamma_0=\Gamma'\cup\ell_0$ as in Propositions \ref{E0 in} and \ref{smooth in E0} and consider the extensions of $\mathcal{O}_{\p3}$-sheaves of the form
\begin{equation}\label{extns on C0}
0 \to \iota_*\mathcal{O}_{\ell_0}(\pt ) \to M \to \iota_*\mathcal{O}_{\Gamma'}((2m-3)\pt ) \to 0.
\end{equation}
We have
$\inext_{\op3}^1(\iota_*\mathcal{O}_{\Gamma'}((2m-3)\pt ),
\iota_*\mathcal{O}_{\ell_0}(\pt ))\simeq\mathbf k(x)$,
where $x=\Gamma'\cap\ell_0$, and $\mathrm{Ext}_{\op3}^1(\iota_*\mathcal{O}_{\Gamma'}((2m-3)\pt),
\iota_*\mathcal{O}_{\ell_0}(\pt ))\simeq\mathbf k$. There is a universal family of extensions
(\ref{extns on C0}) over the affine line $\mathbb{A}^1=\mathbf{V}(\mathrm{Ext}_{\p3}^1(\iota_*\mathcal{O}_{\Gamma'}((2m-3)\pt ),
\iota_*\mathcal{O}_{\ell_0}(\pt ))^{\dual})$ (see \cite[Proposition 3.1]{HH1}):
\begin{equation}\label{univ}
0 \to \boldsymbol{\iota}_*(\mathcal{O}_{\mathbb{A}^1}\boxtimes\mathcal{O}_{\ell_0}(\pt ))\to\boldsymbol{M} \to
\boldsymbol{\iota}_*(\mathcal{O}_{\mathbb{A}^1}\boxtimes\mathcal{O}_{\Gamma'}((2m-3)\pt ))
\to 0,
\end{equation}
where $\boldsymbol{\iota}=id_T\times \iota_0$ and 
$\iota_0:\Gamma_0\hookrightarrow\p3$ is the embedding.

Let $M_t=\boldsymbol{M}|_{\{t\}\times\p3},\ t\in\mathbb{A}^1$. For $t=0$ the extension
(\ref{extns on C0}) splits, i.e.
\begin{equation}\label{L0}
M_0\simeq \iota_*\mathcal{O}_{\Gamma'}((2m-3)\pt )\oplus \iota_*\mathcal{O}_{\ell_0}(\pt ),
\end{equation}
and we may rewrite the triple (\ref{E0 triple}) in the form
\begin{equation}\label{E0 triple 0}
0 \to E_0\to E\overset{s} \to M_0 \to 0.
\end{equation}
On the other hand, for any $t\in\mathbb{A}^1\setminus\{0\}$, the sheaf $M:=M_t$ is a locally 
free $\mathcal{O}_{\Gamma_0}$-sheaf fitting in the exact triple (\ref{extns on C0}).
Take an arbitrary open subset $T$ of $\mathbb{A}^1$ containing 0.
Then the curve $T$ and the sheaves $\mathbf{F}:=\mathcal{O}_T\boxtimes E$ and $\mathbf{G}:=\boldsymbol{M}$ satisfy the hypothesis of Lemma \ref{F,G},
hence by this Lemma, after possibly shrinking $T$,
the epimorphism $s$ in (\ref{E0 triple 0}) extends to an epimorphism
$\mathbf{s}:\mathcal{O}_T\boxtimes E\twoheadrightarrow\boldsymbol{M}$. We thus obtain
the  exact triple
$$0 \to \boldsymbol{E}_T\to\mathcal{O}_T\boxtimes E\xrightarrow{\mathbf{s}}\boldsymbol{M} \to 0.$$
Restricting this triple to any point $t\in T$ and denoting
$E_t=\boldsymbol{E}_T|_{\{t\}\times\p3}$, we obtain an exact triple
\begin{equation}\label{Et}
0 \to E_t\to E\to M_t \to 0,\ \ \ t\in T.
\end{equation}
By Corollary \ref{rk2-mu-stable} and Lemma \ref{stable2}, $E_t$ is stable for any $t\in T$. We thus have a well-defined morphism $\phi:T\to\mathcal{M}(n)$ given by
$t\mapsto[E_t]$.

For $t\neq 0$, the sheaf $M_t$ is a locally free $\mathcal{O}_{\Gamma_0}$-sheaf fitting in
(\ref{extns on C0}). Hence
\begin{equation}\label{res L}
M_t|_{\Gamma'}\simeq \iota_*\mathcal{O}_{\Gamma'}((2m-3)\pt ),\ \ \
M_t|_{\ell_0}\simeq \iota_*\mathcal{O}_{\ell_0}(2\pt ).
\end{equation}

Next, by Proposition \ref{smooth in E0}
$[E_0]$ is a smooth point of $\mathcal{M}(n)$. 
In addition, since the triples (\ref{E0 triple}) and
(\ref{E0 triple 0}) coincide, it follows that $[E_0]\in\partial\cali(n)$ by Proposition \ref{E0 in}. Thus,
\begin{equation}\label{Et in}
[E_0]\in\partial\cali(n)\setminus\mathrm{Sing}\,\mathcal{M}(n).
\end{equation}

Fix a point $t_1\in T\setminus\{0\}$, denote ${\tilde{M}}_0:=M_{t_1}$,
${\tilde{E}}_0:=E_{t_1}$, and rewrite the triple (\ref{Et}) for $t=t_1$:
\begin{equation}\label{tildeE0}
0 \to {\tilde{E}}_0\to E \overset{\varepsilon}\to {\tilde{M}}_0 \to 0.
\end{equation}
Now, as in Lemma \ref{res-trivial}, one can see that
there exists a family of curves $\{\Gamma_b\}_{b\in B}$, parametrized by a curve $B$ with a marked point $0\in B$, such that 
$\Gamma_0=\Gamma'\cup\ell_0$ and $\Gamma_b$ is smooth for $b\in B\setminus\{0\}$, 
and such that ${\tilde{M}}_0$ deforms into an invertible sheaf ${\tilde{M}}_b$ of degre $2m-1$ on $\Gamma_b$ for $b\in B\setminus\{0\}$.
Denote by $\iota_b$ the embedding $\Gamma_b\hookrightarrow\p3$ 
for $b\in B$.
Let $\boldsymbol{\Gamma}\to B$ be the family $\{\Gamma_b\}_{b\in B}$ and 
$\boldsymbol{\iota}=\{\iota_b\}_{b\in B}:
\boldsymbol{\Gamma}\hookrightarrow B\times\p3$ the embedding; for $b=0$, we have the embedding $\iota_0:{\Gamma}_0\hookrightarrow\p3$ defined earlier.
The family of sheaves $\{\tilde{M}_b\}_{b\in B}$ constitutes
an invertible $\mathcal{O}_{\boldsymbol{\Gamma}}$-sheaf
$\boldsymbol{\tilde{M}}$, so that  $\boldsymbol{\iota}_*\boldsymbol{\tilde{M}}$ is
a $\mathcal{O}_{B\times\p3}$-sheaf, flat over $B$, and
the sheaves
$\tilde{M}_b:=\boldsymbol{\iota}_*\boldsymbol{\tilde{M}}|_{\{b\}\times\p3},\ b\in B$, satisfy the relations
\begin{equation}\label{Mb for b ne 0}
{\tilde{M}}_0=M_{t_1},\ \ \ \ 
\tilde{M}_b\simeq \iota_{b*}(\mathcal{O}_{\Gamma_b}(2m-3)
\mathrm{pt}),  \ \  b\in B\setminus\{0\}.
\end{equation}
Since $E|_{\Gamma_0}\simeq\mathcal{O}_{\Gamma_0}^{\oplus2}$ by (\ref{E restr 2}), 
we may assume, after shrinking $B$ if necessary,  that
\begin{equation}\label{E|C0}
E|_{\Gamma_b}\simeq\mathcal{O}_{\Gamma_b}^{\oplus2}, \ \ \ b\in B.
\end{equation}
Formulas (\ref{Mb for b ne 0}) and (\ref{E|C0}) show that the curve $T=B$ and the sheaves
$\mathbf{F}:=\mathcal{O}_B\boxtimes E$ and $\mathbf{G}:=\boldsymbol{\iota}_*\boldsymbol{\tilde{M}}$ satisfy the 
hypothesis of Lemma \ref{F,G}. Hence by this Lemma, after possibly shrinking $B$,
the epimorphism $\varepsilon$ in (\ref{E0 triple 0}) extends to an epimorphism
$\boldsymbol{\varepsilon} : \mathcal{O}_B\boxtimes E
\twoheadrightarrow\boldsymbol{\iota}_*\boldsymbol{\tilde{M}}$, providing a short exact sequence
$$ 0 \to \boldsymbol{\tilde{E}} \to \mathcal{O}_B\boxtimes E \overset{\boldsymbol{\varepsilon}} \to \boldsymbol{\iota}_*\boldsymbol{\tilde{M}} \to 0. $$
Restricting it to any $b\in B$ and denoting
$\tilde{E}_b=\boldsymbol{\tilde{E}}|_{\{b\}\times\p3}$,
we obtain a short exact sequence
\begin{equation}\label{tilde Eb}
0 \to \tilde{E}_b\to E \to \tilde{M}_b\to0,\ \ \ b\in B.
\end{equation}
Since $[E]\in\cali(n-m)$, we deduce from (\ref{Et}) that $\tilde{E}_b$ is stable for any
$b\in B$ by Corollary \ref{rk2-mu-stable} and Lemmas \ref{stable1}, \ref{stable2}. We thus have a well-defined morphism $\phi:B\to\mathcal{M}(n)$, given by $b\mapsto[\tilde{E}_b]$. Now 
(\ref{Mb for b ne 0})-(\ref{tilde Eb}) imply that $[\tilde{E}_b]\in\cald(m,n)$ 
for $0\ne b\in B$. Hence
$B\subset\overline{\cald(m,n)}$. In particular, 
$[E_{t_1}]=[\tilde{E}_0]\in\overline{\cald(m,n)},\ t_1\in
T\setminus\{0\}$. Therefore, $T\subset\overline{\cald(m,n)}$ and, in particular, $[E_0]\in\overline{\cald(m,n)}$. 
This together with (\ref{Et in}) yields
$$ \overline{\cald(m,n)}\cap\left(\partial\cali(n)\setminus\mathrm{Sing}\mathcal{M}(n)\right)
\ne\varnothing . $$
From the irreducibility of $\overline{\cald(m,n)}$, we deduce that
$\overline{\cald(m,n)}\subset\partial\cali(n)$.
\end{proof}

We are finally ready to complete the proof of the main result of this section.

\begin{theorem}\label{Thm 4.6}
For each $n\ge2$ and each $m=1,\dots,n-1,$ $\overline{\cald(m,n)}\subset\partial\cali(n)$.
\end{theorem}

\begin{proof}
The result follows by induction on $m$. For $m=1$, the assertion is true by Proposition \ref{Hooft}.
The induction step $m-1\rightsquigarrow m$ is provided by Proposition \ref{Dmn in In}.
The theorem is proved.
\end{proof}

\section{Elementary transformations along elliptic quartic curves} \label{cicurves}

In this section, we consider the case in which the curve $\Sigma$ (notation from Section \ref{degen}) is an elliptic quartic curve, that is a complete intersection of two hypersurfaces $f_1=0$ and $f_2=0$ of degree $2$. The minimal locally free resolution of its structure sheaf has the form
\begin{equation}\label{resolution}
0 \to \op3(-4) \to \op3(-2)^{\oplus 2} \to \op3 \to \iota_*\osigma \to 0,
\end{equation}
where $\iota:\Sigma \to \p3$ is the inclusion map.

Following the procedure outlined by Proposition \ref{degdata}, let
$L\in{\rm Pic}^0(\Sigma)$. If $L$ is nontrivial, we have $h^p(\iota_*L)=h^p(\Sigma,L)=0$, for $p=0,1$, as desired.

Let $E$ be a locally free instanton sheaf of rank $2$ and charge $n$. In order to perform an elementary transformation of $E$ along $\Sigma$, we must find out whether there exists a surjective map $E\to (\iota_*L)(2)$. In Lemma \ref{hom e call} below, we give the affirmative answer in the case when $E$ is a null-correlation bundle, that is an instanton of charge 1.

Breaking (\ref{resolution}) into short exact sequences and tensoring by any instanton bundle $E$, we obtain:
$$ 0 \to E(-4) \to E(-2)^{\oplus 2} \to E\otimes I_\Sigma \to 0 ~~ {\rm and} ~~
0 \to E\otimes I_\Sigma \to E \to E|_\Sigma \to 0 , $$
where $I_\Sigma$ denotes the ideal sheaf of $\Sigma$.

Since $h^1(E(-2))=h^2(E(-2))=0$, we conclude from the first sequence that $H^1(E\otimes I_\Sigma)\simeq H^2(E(-4))$ for every elliptic quartic curve $\Sigma$; moreover, if $E$ is not the trivial instanton, then also $h^2(E\otimes I_\Sigma)=0$.  Now, moving to the second sequence above, we obtain for a nontrivial instanton $E$:
$$
0 \to H^0(E|_\Sigma) \to H^2(E(-4)) \to H^1(E) \to H^1(E|_\Sigma) \to 0.
$$
It is not difficult to check that $h^2(E(-4))=h^1(E)=2n-2$.

\begin{lemma}\label{hom e call}
If $E$ is a null-correlation bundle on $\p3$, then its restriction to any nonsingular elliptic
quartic curve $\Sigma\hookrightarrow\p3$ is either of the form $L_0\oplus L_0^{\dual}$ for some 
nontrivial $L_0\in{\rm Pic}(\Sigma)$, or the nontrivial extention of a nontrivial
$L_0\in{\rm Pic}^0(\Sigma)$ of order $2$ by itself. In either case, we have
$$ \Hom(E,(\iota_*L)(2)) \simeq H^0(\Sigma,L_0\otimes L(8\pt))\oplus
H^0(\Sigma,L_0^{\dual}\otimes L(8\pt)). $$
\end{lemma}

\begin{proof}
The restriction $E|_\Sigma$ is a rank $2$ bundle on $\Sigma$ with trivial determinant; if $E$ is a null-correlation bundle, then also $h^0(\Sigma,E|_\Sigma)=h^1(\Sigma,E|_\Sigma)=0$. The first claim now follows from Atiyah's classification of rank $2$ bundles on nonsingular elliptic curves.

Next, note that $\Hom(E,(\iota_*L)(2))\simeq H^0(\Sigma,E^\dual|_\Sigma\otimes L(8\pt))$. If
$E|_\Sigma\simeq L_0 \oplus L_0^{\dual}$, the second claim follows immediately. On the orther hand, if $E|_\Sigma$ is an extension of the form
$$ 0 \to L_0 \to E|_\Sigma \to L_0 \to 0 ~~{\rm with}~~ L_0^2\simeq\osigma, $$
then, twisting by $L(8\pt)$, we obtain the cohomology exact sequence
$$ 0 \to H^0(\Sigma,L_0\otimes L(8\pt)) \to H^0((E|_\Sigma)^\dual\otimes L(8\pt)) \to
H^0(\Sigma,L_0\otimes L(8\pt)) \to 0 , $$
and the second claim also follows.
\end{proof}

In particular, we conclude that for any null-correlation bundle $E$ and any elliptic quartic curve
$\iota:\Sigma\hookrightarrow\p3$, there exists a surjective map
$\varphi:E\to (\iota_*L)(2)$. The kernel sheaf $E':=\ker\varphi$ is a rank $2$ instanton of charge $5$, since $\Sigma$ has degree $4$; note, in addition, that $E'$ is $\mu$-stable by Lemma \ref{mu-semistable}, thus in particular $[E']\in\calm(5)$.

Our next goals are evaluating the dimension and proving
the generic smoothness of the locus of all the instanton
sheaves obtained by elementary transformations 
from null-correlation bundles along elliptic quartic curves.

Let $\cale_4$ denote the set of nonsingular elliptic quartic curves in $\p3$. It can be regarded
as an open subset of the Grassmannian $G(2,S^2V)$, where V is the $4$-dimensional complex vector 
space such that $\p3=\mathbb{P}(V)$. This is a family of nonsingular elliptic curves, so let
$j:\calj_4\to \cale_4$ denote the relative Jacobian variety, i.e.
$j^{-1}(\Sigma)={\rm Pic}^0(\Sigma)$, and denote
$\calj^{\rm o}_4:=\calj_4 \setminus \{{\rm zero~~section}\}$. 
A point of $\calj^{\rm o}_4$ can be thought of as a pair $(\Sigma,[L])$, in which $\Sigma$ is a
smooth elliptic quartic curve and $L$ is a nontrivial line bundle of degree $0$ on $\Sigma$, so that $[L]\in{\rm Pic}^0(\Sigma)$. Equivalently, it can be thought of as the isomorphism class of the sheaf $\iota_*L$ on $\p3$, where $\iota:\Sigma\hookrightarrow\p3$ is the natural embedding.

Note that $\calj^{\rm o}_4$ is an irreducible, quasiprojective variety of dimension $17$.

Consider now the following subset of $\calm(5)$:
$$
\calq_5 := \{ [E]\in\calm(5) ~|~ [E^{\dual\dual}]\in\cali(1), ~~
[(E^{\dual\dual}/E)(-2)]\in \calj_4^{\rm o}\}.
$$
Let $\overline{\calq_5}$ denote the closure of $\calq_5$ in $\calm(5)$.

\begin{theorem}\label{thm-quartics}
$\overline{\calq_5}$ is an irreducible component of $\calm(5)$ of dimension $37$, distinct from the instanton component $\overline{\cali(5)}$.
\end{theorem}

\begin{proof}
Consider the map $\varpi:\calq_5\to\cali(1)\times \calj_4^{\rm o}$ given by
$E \mapsto \big([E^{\dual\dual}], [(E^{\dual\dual}/E)(-2)]\big)$, and observe that it is surjective. Indeed, given a null-correlation bundle $F\in\cali(1)$, and a nonsingular elliptic quartic curve $\iota:\Sigma\hookrightarrow\p3$ equipped with a nontrivial $L\in{\rm Pic}^0(\Sigma)$, so that $[\iota_*L]\in \calj^{\rm o}_4$, we know from Lemma \ref{hom e call} that there exists a surjective map $\varphi:F\to(\iota_*L)(2)$. Then, as we have seen above, $E:=\ker\varphi$ is an instanton sheaf from $\calm(5)$ such that $E^{\dual\dual}\simeq F$ and $E^{\dual\dual}/E\simeq (\iota_*L)(2)$, i.e. $[E]\in\calq_5$.

The fiber $\varpi^{-1}(F,\iota_*L)$ consists precisely of all surjective maps $F\to(\iota_*L)(2)$ up to homothety, so it forms an open subset of $\mathbb{P}(\Hom(F,(\iota_*L)(2)))$, which, again by Lemma \ref{hom e call}, has dimension $15$ for every $[F]\in\cali(1)$ and every $[\iota_*L]\in \calj_4^{\rm o}$.

It follows that $\calq_5$ is an irreducible quasiprojective variety of dimension
$$ \dim \cali(1) + \dim \calj_4^{\rm o} + \dim \mathbb{P}\Hom(F,(\iota_*L)(2)) = 37. $$
We have concluded that $\overline{\calq_5}$ is an irreducible subvariety of $\calm(5)$ of the same dimension as $\overline{\cali(5)}$ and whose generic point represents a non locally free instanton sheaf. This implies that $\overline{\calq_5}$ and $\overline{\cali(5)}$ are distinct components of $\calm(5)$.
\end{proof}

In particular, we have the following interesting consequence.

\begin{corollary}
The moduli space of instanton sheaves of rank $2$ and charge $5$ is reducible.
\end{corollary}

\begin{remark}
It is not clear, however, whether the moduli space of instanton sheaves of rank $2$ and charge $5$ is connected; in other words, we do not know whether the intersection
$\overline{\calq}\cap\overline{\cali(5)}$ contains instanton sheaves.
\end{remark}

Finally, it is interesting to note that the generic point of $\overline{\calq}$ is a smooth point of~$\calm(5)$.

\begin{proposition}\label{q=smooth pt}
Let $F$ be an instanton sheaf obtained, as above, by an elementary transformation of a null-correlation bundle along a pair $(\Sigma,L)$ where $\Sigma$ is an elliptic quartic curve in $\p3$ and $[L]\in \mathcal{J}^{\rm o}_4$,
$L^2\not\simeq\mathcal{O}_{\Sigma}$. Then $\dim\ext^1(F,F) = 37$ and $\dim\ext^2(F,F) = 0$, so $[F]$ is a smooth point of $\calm(5)$.
\end{proposition}

\begin{proof}
First, note that, as $\Sigma$ is a complete intersection of two quadrics, the isomorphism (\ref{id30}) yields
\begin{equation}\label{id300}
\inext^1(F,(\iota_*L)(2))\simeq(\iota_*\mathcal{O}_{\Sigma})(4).
\end{equation}
Next, applying the functor $\inhom(-,E)$ to the sequence 
$$ 0 \to F \to E \to (\iota_*L)(2) \to 0, $$
where $E:=F^{\dual\dual}$ is a null-correlation bundle, and using the isomorphism
$\inext^2(\mathcal{O}_{\Sigma},\op3)\simeq\iota_*\det N_{\Sigma/\p3}\simeq(\iota_*\mathcal{O}_{\Sigma})(4)$, we obtain
\begin{equation}\label{id301}
\inext^1(F,E)\simeq\inext^2((\iota_*L)(2),E)\simeq
(\iota_*L^{-1})(-2)\otimes E|_{\Sigma}\otimes\inext^2(\mathcal{O}_{\Sigma},\op3)
\end{equation}
$$
\simeq (\iota_*L^{-1})(2)\otimes E|_{\Sigma}.
$$
Next, using the isomorphism $\intor_1((\iota_*L)(2),\iota_*\mathcal{O}_{\Sigma})\simeq
(\iota_*L)(2)\otimes N_{\Sigma/\p3}^{\dual}\simeq\iota_*L^{\oplus2}$,
similarly to (\ref{sqc2}), we obtain an exact sequence
$$ 0 \to \iota_*L^{\oplus2} \to F|_{\Sigma} \to E|_{\Sigma} \to (\iota_*L)(2) \to 0. $$
Since $\det E|_{\Sigma}\simeq\mathcal{O}_{\Sigma}$, it follows that
\begin{equation}\label{ker1}
\ker\left\{E|_{\Sigma}\twoheadrightarrow(\iota_*L)(2)\right\} \simeq (\iota_*L^{-1})(-2).
\end{equation}
Hence the above exact sequence yields the following exact triple:
$$ 0 \to \iota_*L^{\oplus2}\to F|_{\Sigma} \to (\iota_*L^{-1})(-2) \to 0. $$
As $[\iota_*L]\in \mathcal{J}^{\rm o}_4$, it follows that $h^1((\iota_*L^2)(2))=h^1(\Sigma,L^2(8\pt))=0$. Therefore the last sequence splits, i.e. $F|_{\Sigma}\simeq\iota_*L^{\oplus2}\oplus(\iota_*L^{-1})(-2)$. This yields
$\inhom(F,(\iota_*L)(2))\simeq(\iota_*\mathcal{O}^{\oplus2}_{\Sigma})(2)\oplus(\iota_*L^2)(4)$.
Substituting this relation together with (\ref{id300}) and (\ref{id301}) into (\ref{inhom sqc 2}), we obtain the exact sequence
\begin{multline*}
0\to(\iota_*L^2)(4)\to\iota_*\mathcal{O}_{\Sigma}(2)^{\oplus2}\oplus(\iota_*L^2)(4)
\to\inext^1(F,F)\to\\ (\iota_*L^{-1})(2)\otimes E|_{\Sigma}\to(\iota_*\mathcal{O}_{\Sigma})(4)\to0.
\end{multline*}
Similarly to (\ref{ker1}), one has
$\ker\left\{ (\iota_*L^{-1})(2)\otimes E|_{\Sigma} \twoheadrightarrow(\iota_*\mathcal{O}_{\Sigma})(4) \right\} \simeq \iota_*L^{-2}$.
Thus the last exact sequence provides the exact triple
$$ 0 \to (\iota_*\mathcal{O}_{\Sigma})(2)^{\oplus2} \to \inext^1(F,F) \to \iota_*L^{-2} \to 0. $$
Moreover, since $h^1((\iota_*L^{2})(2)^{\oplus2})=0$, this triple splits, and 
$$ \inext^1(F,F) \simeq (  \iota_*\mathcal{O}_{\Sigma})(2)^{\oplus2}\oplus\iota_*L^{-2}. $$
Since by assumption $L^{-2}\not\simeq\mathcal{O}_{\Sigma}$, it follows that $h^1(\Sigma,L^{-2})=0$,
and the above isomorphism implies $h^1(\inext^1(F,F))=0$. This together with Lemma \ref{global exts} yields the Proposition.
\end{proof}

We conclude by remarking that one can also perform elementary transformations of the trivial bundle along plane cubic curves, a fact also realized by Perrin in \cite{Per3,Per4}, as mentioned in the Introduction. 

Indeed, let $\iota:\Sigma \to \p3$ be a plane cubic curve. For any $L\in{\rm Pic}^0(\Sigma)$ one can find surjective morphisms $\varphi:\op3^{\oplus2}\to (\iota_*L)(2)$, since $\Hom(\op3^{\oplus2},(\iota_*L)(2))\simeq H^0(\Sigma,L(6\pt))^{\oplus2}$. Assuming that $L$ is nontrivial, the kernel sheaf $E':=\ker\varphi$ is an instanton sheaf of charge $3$; it is also stable, by Lemma \ref{stable1}, therefore $[E']\in\calm(3)$. 

\begin{remark}
One can also consider surjective morphisms $\varphi:\op3^{\oplus2}\to (\iota_*\osigma)(2)$. In this case, the kernel $E':=\ker\varphi$ is a stable rank $2$ torsion free sheaf with $c_1(E')=c_3(E')=0$ and $c_2(E')=3$, but is not an instanton sheaf. 
\end{remark}

Next, let now $\cale_3$ be the set of nonsingular plane cubic curves, regarded as an open subset of
$\mathbf{P}(S^3(T_{(\p3)^\vee}(-1)))$,
let $\calj_3$ be the relative Jacobian over $\cale_3$, and  $\calj^{\rm o}_3$ the complement of the zero section. Consider the following set:
$$ \calq_3 := \{ [E]\in\calm(3) ~|~ E^{\dual\dual}\simeq\op3^{\oplus2}, ~~ 
[(E^{\dual\dual}/E)(-2)]\in \calj_3^{\rm o}\}. $$
Similarly to Theorem \ref{thm-quartics}, one can show that $\overline{\calq_3}$ is an irreducible component of $\calm(3)$ of dimension
$$ \dim \calj_3^{\rm o} + h^0(\Sigma,L(6\pt)^{\oplus2}) - \dim {\rm Aut}(\op3^{\oplus2}) = 21 .$$
It follows that $\overline{\calq_3}$ does not coincide with the instanton component $\overline{\cali(3)}$, since $\dim\cali(3)=21$ as well. 

The proof of Proposition \ref{q=smooth pt} also works in this case, and one can show that the instanton
sheaves obtained as above, by elementary transformations of the trivial rank 2 bundle along a
pair $(\Sigma,L)$ where $\Sigma$ is a plane cubic curve in $\p3$ and $[L]\in \calj^{\rm o}_3$,
$L^2\not\simeq\mathcal{O}_{\Sigma}$, are smooth points of $\calm(3)$. 

Furthermore, as mentioned at the Introduction, Perrin provides additional information on the intersection $\overline{\calq_3}\cap\overline{\cali(3)}$ of these two irreducible components of $\calm(3)$. In fact, he has shown that if $L$ is a theta characteristic on $\Sigma$ (i.e. if $L^2=\mathcal{O}_{\Sigma}$), then the sheaves $E$ given by short exact sequences of the form
$$ 0 \to E \to \op3^{\oplus3} \to (\iota_*L)(2) \to 0 $$
lie in the closure of the instanton component $\overline{\cali(3)}$, and form an irreducible component of the instanton boundary $\partial\cali(3)$, see \cite[Thm. 0.1]{Per3}.

\end{document}